\documentclass{amsart}
\usepackage{amssymb}
\usepackage{amsmath}
\usepackage{amsthm}
\usepackage{amscd}
\usepackage{amstext}
\usepackage{amsfonts}
\usepackage[all]{xy}
\usepackage{graphicx}

\theoremstyle{plain}
\newtheorem{theorem}{Theorem}[section]

\newtheorem{lemma}[theorem]{Lemma}
\newtheorem{proposition}[theorem]{Proposition}
\newtheorem{assumption}[theorem]{Assumption}
\theoremstyle{definition}
\newtheorem{definition}[theorem]{Definition}
\theoremstyle{remark}
\newtheorem{remark}[theorem]{Remark}
\newtheorem{example}[theorem]{Example}

\newcommand{\Z}{\mathbb{Z}}
\newcommand{\R}{\mathbb{R}}
\newcommand{\C}{\mathbb{C}}
\newcommand{\T}{\mathbb{T}}
\newcommand{\D}{\mathbb{D}}
\newcommand{\SSS}{\mathbb{S}}
\newcommand{\cpt}{\mathrm{cpt}}
\newcommand{\pt}{\mathrm{pt}}
\newcommand{\id}{\mathrm{id}}
\newcommand{\diag}{\mathrm{diag}}
\newcommand{\tT}{\widetilde{\T}}
\newcommand{\tS}{\widetilde{\SSS}}
\newcommand{\tH}{\widetilde{H}}
\newcommand{\thh}{\tilde{h}}
\newcommand{\CP}{\mathbb{C}\mathrm{P}}
\newcommand{\uC}{\underline{\C}}
\newcommand{\Kar}{\mathrm{Kar}}
\newcommand{\Bloch}{E_B}
\newcommand{\eBloch}{\widetilde{E}_B}

\DeclareMathOperator{\ind}{\mathrm{index}}
\renewcommand{\sf}{\mathrm{sf}}
\DeclareMathOperator{\Ker}{\mathrm{Ker}}

\DeclareMathOperator{\Ad}{\mathrm{Ad}}

\begin{document}
\title[Index theory and second-order bulk-boundary correspondence]
{Index theory and bulk-boundary correspondence for inversion-symmetric second-order topological insulators}

\author[S. Hayashi]{Shin Hayashi}
\address{Department of Mathematical Sciences, Aoyama Gakuin University, 5-10-1 Fuchinobe, Chuo, Sagamihara, Kanagawa 252-5258, Japan}
\address{Advanced Institute for Materials Research, Tohoku University, 2-1-1 Katahira, Aoba, Sendai 980-8577 Japan}
\address{RIKEN Center for Interdisciplinary Theoretical and Mathematical Sciences(iTHEMS), RIKEN, Wako 351-0198, Japan}

\email{{\tt shin.hayashi@math.aoyama.ac.jp}}

\begin{abstract}
We provide an index-theoretic proof of the bulk-boundary correspondence for two- and three-dimensional second-order topological insulators that preserve inversion symmetry, which are modeled as rectangles and rectangular prism-shaped systems.
Our method uses extensions of the symbols of some Toeplitz operators on discrete quarter planes and computations of topological equivariant $K$-theory groups.
\end{abstract}

\maketitle
\setcounter{tocdepth}{1}
\tableofcontents

\section{Introduction}
In this paper, we discuss the bulk-boundary correspondence for some topological crystalline insulators.
Focusing on inversion-symmetric second-order topological insulators, we provide an index-theoretic proof for the correspondence.

The bulk-boundary correspondence is a general principle that is investigated in condensed matter physics.
For (first-order) topological insulators, topological codimension-one boundary states (edge states and surface states) appear corresponding to some gapped bulk topological invariants \cite{Hat93b,KRSB02,GP13,PSB16}.
Higher-order topological insulators (HOTIs) have higher codimensional boundary states (corner states and hinge states), and the bulk-boundary correspondence is much more generalized as relationships between gapped topological invariants and higher codimensional boundary states \cite{BBH17a,Khalaf18a,Schindler18}.
The codimension corresponds to the order, and our focus on second-order phases leads to discussions on two-dimensional systems with corners and three-dimensional systems with hinges.
HOTIs are divided into two classes \cite{Geier18} in accordance with the dependence of the topology characterizing them on boundary conditions (or lattice terminations).
{\em Extrinsic} HOTIs depend on boundary conditions.
{\em Intrinsic} HOTIs are those that are independent of the imposed boundary conditions and characterized by their bulk properties.
For intrinsic HOTIs, it is well known that some spatial symmetry (as point group symmetry) is needed.
Expanding the studies on topological crystalline insulators \cite{FK07,HPB11,SS14}, many results have been obtained, including their classification \cite{Geier18,Khalaf18b}, the topological invariants characterizing them \cite{BBH17a,Khalaf18a,KSS20}, and the correspondence between topological invariants and higher codimensional boundary states \cite{LB19,TTM20a,TTM20b}.

Mathematical studies of topological insulators were initiated by Bellissard \cite{BvES94}, and Kellendonk--Richter--Schultz-Baldes provided a proof of the bulk-boundary correspondence based on index theory for Toepltiz operators \cite{KRSB02}.
$K$-theory is employed to classify topological insulators \cite{Kit09}, and it has been widely expanded to $K$-theoretic studies of topological crystalline insulators \cite{SS14,Ku15,SSG17,GT19a,OSS19,GKT21}.
Regarding the bulk-boundary correspondence of topological crystalline insulators, Gomi--Thiang conducted an index-theoretic study under glide reflection symmetry \cite{GT19b}.
The bulk-interface correspondence for one-dimensional inversion-symmetric topological insulators was investigated by Thiang--Zhang \cite{TZ23}.
The topological aspects of higher codimensional boundary states were studied in \cite{Hayashi2} based on index theory for Toeplitz operators on discrete quarter planes \cite{Sim67,DH71,CDS72,Dud77b,Pa90}.
In \cite{Hayashi3}, this index theoretic approach was applied to the Benalcazar--Bernevig--Hughes model of higher-order topological insulators \cite{BBH17a}, which involves a single corner and does not require point group symmetry.
The topological invariants introduced there depend on boundary conditions and therefore correspond to extrinsic HOTIs.
By expanding the above method within the 10-fold way classification scheme, an index theory-based classification table for topological corner states was proposed in \cite{Hayashi4}, and the results were consistent with the classification table for extrinsic HOTIs in \cite{Geier18}.
Regarding the index-theoretic approach to intrinsic HOTIs, point group symmetry should be incorporated into the framework.
Such a framework was established by Ojito--Prodan--Stoiber \cite{OPS24a}, where the higher-order bulk-boundary map relating bulk invariants and higher codimensional boundary modes was derived via a spectral sequence argument.
Its nontriviality shows the higher-order bulk-boundary correspondence, which is provided by their space-adiabatic method \cite{OPS24b}.
They confirmed the correspondence across various classes, including three-dimensional inversion-symmetric second-order topological insulators \cite{OPS24a}.

In this paper, we present an alternative approach for the bulk-boundary correspondence of inversion-symmetric second-order topological insulators.
We consider both two- and three-dimensional systems modeled on rectangles and rectangular prisms.
In the two-dimensional case, we assume the presence of chiral symmetry whose symmetry operator anticommutes with the inversion symmetry operator.
In these classes, the classification of HOTIs is known to be $\Z/2$ in \cite{Geier18,Khalaf18b},
and their bulk-gapped topological invariants are provided by symmetry-based indicators \cite{Po17,Khalaf18a,MW18,OW18,OTY21,SO21}.
For the bulk-boundary correspondence of inversion-symmetric second-order topological insulators, we follow the formulation presented in \cite{TTM20a,TTM20b}, and we provide an alternative index-theoretic proof (Theorems~\ref{theorem3D} and \ref{theorem2D}).
For the proof, we expand the discussion in \cite{Hayashi5}, which utilized extensions of symbols of Toeplitz operators on discrete quarter planes (in other words, extensions of bulk Hamiltonians).
As in \cite{Hayashi5}, the key ingredient is the Wiener--Hopf factorization for matrix-valued functions, as developed by Gohberg, Kre{\u \i}n \cite{GK58r,CG81,GKS03}.
Note that the interior of the unit circle on which we extend the symbols was discussed by Graf--Porta in their study of first-order bulk-boundary correspondences \cite{GP13}, which inspires \cite{Hayashi5}.
We also note that, for inversion-symmetric topological insulators, detailed studies were conducted in \cite{HPB11}.
We model rectangles and rectangular prism-shaped systems, although we employ infinite lattices to discuss corners or hinges, as in \cite{Hayashi2}.
To incorporate the inversion symmetry that interchanges different corners or hinges within the framework, we follow \cite{OPS24a} and consider several infinite lattices simultaneously.

In three-dimensional systems, for example, we consider translation-invariant single-particle Hamiltonians on the lattice $\Z^3$ that preserve inversion symmetry.
We assume that the bulk Hamiltonian on the Brillouin torus and our models for four surfaces are gapped at zero energy (Assumptions~\ref{assumption3D} and \ref{assumption2D}).
Under this setup, as in \cite{Hayashi5}, we see that the bulk Hamiltonian on the three-dimensional torus canonically extends to some space preserving the gap.
This extended Hamiltonian preserves an extension of the inversion symmetry and provides an element of some $\Z/2$-equivariant topological $K$-theory group.
We next relate this $K$-class invariant to the symmetry-based indicator and to the number of chiral hinge states for the four hinges of a rectangular prism.
The discussion is based on calculations of equivariant $K$-groups and the maps between them, which leads to the second-order bulk-boundary correspondence.
For two-dimensional cases, we use $K_\pm$-theory \cite{Witten98,AH04} and prove the correspondence in a similar manner.
The two-dimensional and three-dimensional cases are related through the Thom isomorphism.
We start from the three-dimensional case, and the two-dimensional case follows from the relation between the $K_{\Z/2}$- and $K_\pm$-theories.
For the computations of equivariant $K$-theory groups and the maps between them, we mainly follow Gomi \cite{Gomi15}.

One advantage of our approach is the elementary nature of our discussion.
We extend bulk Hamiltonians by using edge- or surface-gapped conditions and take their $K$-class in the {\em topological} equivariant $K$-theory group.
Correspondingly, our computations of $K$-groups and the maps between them are conducted within topological $K$-theory rather than that for $C^*$-algebras.
While computational difficulties persist in applications involving explicit models, as in \cite{Hayashi5}, the Wiener--Hopf factorization offers a strategy for verifying the edge- or surface-gapped condition for explicit model Hamiltonians.
Following this method, nontrivial (mathematical) examples are presented in Sect.~\ref{Sect5}.
Compared with the $C^*$-algebraic framework of Ojito--Prodan--Stoiber \cite{OPS24a}, our discussion is restricted: our models for codimension-one boundaries (edges and surfaces) are compressions of bulk Hamiltonians, and perturbations around the boundary are not allowed.
We also note that the shapes of the systems discussed in this paper (rectangles and rectangular prisms) are restricted in comparison with the general principle of the bulk-boundary correspondence for intrinsic HOTIs in condensed matter physics, and the result is expected to be expanded further to systems with various shapes, such as parallelograms.
These restrictions originate from the reliance on the index formula presented in \cite{Hayashi5}.
Although the relationship between the $C^*$-algebraic framework \cite{OPS24a} and our approach is not clear, in place of the higher differentials of the spectral sequence in \cite{OPS24a}, we lift the $K$-class invariants for both the bulk and corner or hinge invariants in topological $K$-groups.
Regarding the contact between index theory for quarter-plane Toeplitz operators and a topic in noncommutative geometry,
we note that the pullback $C^*$-algebra containing the pair of half-plane Toeplitz operators with the same symbol in \cite{Pa90} is known as the algebra of a quantum three-sphere \cite{BHMS05}.

This paper is organized as follows.
In Sect.~\ref{Sect2}, we collect the necessary results and notations used in this paper.
In Sect.~\ref{Sect3}, we consider three-dimensional systems.
Following \cite{TTM20a}, we formulate the second-order bulk-boundary correspondence for inversion-symmetric topological insulators and provide its proof.
Two-dimensional systems are discussed in Sect.~\ref{Sect4}.
Since the discussion in Sect.~\ref{Sect4} relies on that presented in Sect.~\ref{Sect3}, some results and notations used in Sect.~\ref{Sect4} are also included in Sect.~\ref{Sect3}.
Examples are contained in Sect.~\ref{Sect5}.

\section{Preliminaries}
\label{Sect2}
In this section, we collect the necessary results and notations that are used in this paper.
One more set of preliminaries from \cite{Gomi15} will be required and is contained in Sect.~\ref{Sect3}.

\subsection{Toeplitz operators}
\label{Sect2.1}
Throughout this paper, we write $\T$ for the unit circle in the complex plane and
$\Z/n$ for the cyclic group of order $n$ (written additively).
Let $f \colon \T \to \C$ be a continuous map.
We consider the multiplication operator $M_f$ on $l^2(\Z)$ generated by $f$.
Let $\Z_{\geq 0}$ (resp. $\Z_{\leq 0})$ be the set of nonnegative (resp. nonpositive) integers.
Let $T^+_f $ be the compression of the multiplication operator $M_f$ onto $l^2(\Z_{\geq 0})$, which is the operator defined by $T^+_f \varphi = P_{\geq 0} M_f \varphi$, where $\varphi$ in $l^2(\Z_{\geq 0})$ and $P_{\geq 0}$ denotes the orthogonal projection of $l^2(\Z)$ onto its closed subspace $l^2(\Z_{\geq 0})$.
Similarly, we write $T^-_f$ for the compression of $M_f$ onto $l^2(\Z_{\leq 0})$.
They are called {\em Toeplitz operators}.
Let $f \colon \T^2 \to \C$ be a continuous map and $M_f$ be the multiplication operator on $l^2(\Z^2)$ generated by $f$.
Let $T^1_f$, $T^2_f$, $T^3_f$ and $T^4_f$ be the compressions of $M_f$ onto $l^2(\Z_{\geq 0} \times \Z)$, $l^2(\Z \times \Z_{\geq 0})$, $l^2(\Z_{\leq 0} \times \Z)$ and $l^2(\Z \times \Z_{\leq 0})$, respectively.
They are called {\em half-plane Toeplitz operators}.
We also consider the compressions onto $l^2(\Z_{\geq 0} \times \Z_{\geq 0})$, $l^2(\Z_{\leq 0} \times \Z_{\geq 0})$, $l^2(\Z_{\leq 0} \times \Z_{\leq 0})$ and $l^2(\Z_{\geq 0} \times \Z_{\leq 0})$, which we denote as $T^a_f$, $T^b_f$, $T^c_f$ and $T^d_f$, respectively. They are called {\em quarter-plane Toeplitz operators}.
In what follows, we may consider systems of these operators but simply call them (half-plane, quarter-plane) Toeplitz operators rather than systems of them for simplicity.
According to Douglas--Howe \cite{DH71}, the quarter-plane Toeplitz operator $T_f^a$ is Fredholm if and only if the half-plane Toeplitz operators $T^1_f$ and $T^2_f$ are invertible.
Note that the intersection of two discrete half planes $\Z_{\geq 0} \times \Z$ and $\Z \times \Z_{\geq 0}$ is the discrete quarter plane $\Z_{\geq 0} \times \Z_{\geq 0}$.
Similarly, the quarter-plane Toeplitz operators $T^b_f$, $T^c_f$ and $T^d_f$ are Fredholm if and only if the half-plane Toeplitz operators $(T^2_f, T^3_f)$, $(T^3_f, T^4_f)$ and $(T^4_f, T^1_f)$ are invertible, respectively.

In what follows, although we discuss operators on infinite lattices, we model systems with four edges/surfaces (codimension-one boundaries) and four corners/hinges (codimension-two boundaries), and their numbering schemes are shown in Fig.~\ref{Fig1}.
\begin{figure}
  \centering
  \includegraphics[width=12cm]{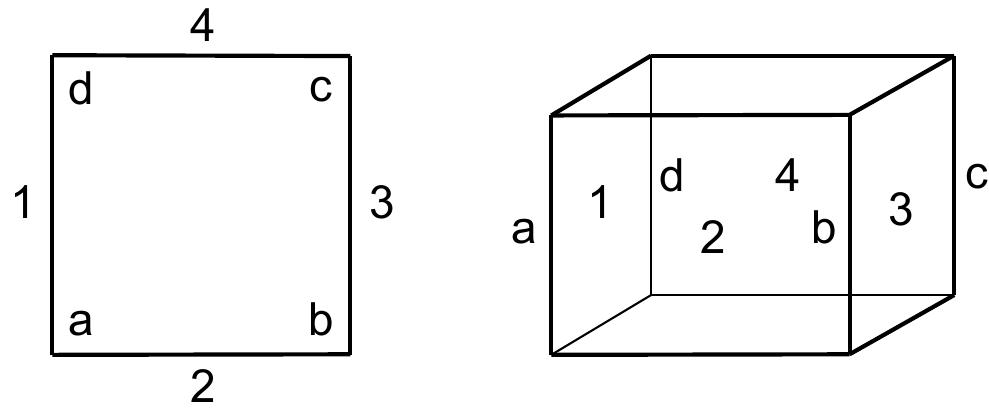}
  \caption{Labeling schemes for four edges and four corners of rectangles (left) and four surfaces and four hinges of rectangular prisms (right)}
    \label{Fig1}
\end{figure}

\subsection{Wiener--Hopf factorizations}
\label{Sect2.2}
We next introduce factorizations of matrix-valued functions called {\em Wiener--Hopf factorizations}.
For the factorizations, we refer the reader to \cite{GK58r,CG81,GKS03}.

Let $D_+ = \{ z \in \C \mid \lvert z \lvert \ < 1\}$ and $D_- = \{ z \in \C \mid \lvert z \lvert \ > 1 \} \cup \{ \infty\}$, which are open disks.
We write $\D_\pm = \T \cup D_\pm$ (the double sign corresponds) for closed disks whose union is the Riemann sphere $\SSS^2 = \C \cup \{ \infty \}$.
For a nonsingular rational matrix function $f \colon \T \to GL(n,\C)$ with poles off $\T$, the following decomposition exists (for the associated algorithm, see \cite{GKS03}):
\begin{equation}
\label{rightWH}
	f = f_- \Lambda f_+,
\end{equation}
where $f_\pm$ and $\Lambda$ are continuous maps $\T \to GL(n,\C)$ satisfying the following.
(a) $\Lambda$ is the diagonal matrix-valued function of the form $\Lambda(z) = \diag(z^{\kappa_1}, \ldots, z^{\kappa_n})$, where $\kappa_1 \geq \ldots \geq \kappa_n$ are nonincreasing sequences of integers called {\em right partial indices}.
(b) $f_+$ admits a continuous extension onto $\D_+$ that is holomorphic on $D_+$ as a nonsingular matrix-valued function.
(c) $f_-$ admits a continuous extension onto $\D_-$ that is holomorphic on $D_-$ as a nonsingular matrix-valued function\footnote{In \cite{PS86}, a similar factorization is referred to as the {\em Birkhoff factorization}.}.
A similar factorization also exists, but the placement of $f_\pm$ is exchanged:
\begin{equation}
\label{leftWH}
	f = f_+ \Lambda f_-.
\end{equation}
We call (\ref{rightWH}) the {\em right} Wiener--Hopf factorization and (\ref{leftWH}) the {\em left} Wiener--Hopf factorization.
The nonincreasing sequences of integers $\kappa_1 \geq \ldots \geq \kappa_n$ in (\ref{leftWH}) are called {\em left} partial indices.
Note that partial indices are independent of the choice of Wiener--Hopf factorization and depend only on the matrix-valued function $f$.
In both the right and left factorizations, when the partial indices are trivial, they are called {\em canonical} factorizations.
For both the right and left Wiener--Hopf factorizations, we use $f^e_+$ and $f^e_-$ to denote the extensions of $f_+$ and $f_-$ in (b) and (c), respectively.
They are continuous maps $f^e_\pm \colon \D_\pm \to GL(n,\C)$, where the double sign corresponds.
There are the following relations between Toeplitz operators and Wiener--Hopf factorizations \cite{CG81}.
\begin{itemize}
\item There exists a right canonical factorization $f = f_- f_+$ if and only if the Toeplitz operator $T^+_f$ on $l^2(\Z_{\geq 0}, \C^n)$ is invertible.
\item There exists a left canonical factorization $f = f_+ f_-$ if and only if the Toeplitz operator $T^-_f$ on $l^2(\Z_{\leq 0}, \C^n)$ is invertible.
\end{itemize}

Index formulas for quarter-plane Toeplitz operators were developed by Coburn--Douglas--Singer \cite{CDS72}, Dudu\v cava \cite{Dud77b}, Park\cite{Pa90}.
Here, we investigate another formula \cite{Hayashi5} that employs extensions of the symbols of quarter-plane Toeplitz operators obtained by using Wiener--Hopf factorizations.
We expand it to an equivariant setup.
According to \cite{Hayashi5}, when $T^+_f$ is invertible, we extend the nonsingular rational matrix function $f$ on $\T$ continuously to $\D_+$ by using a right canonical factorization as follows:
\begin{equation*}
	f^e(z) = f^e_-(\bar{z}^{-1})f^e_+(z),
\end{equation*}
where $z \in \D_+$.
This $f^e$ is a nonsingular matrix-valued function that is independent of the choice of factorization \cite[Lemma~3.1]{Hayashi5}.
Similarly, when $T^-_f$ is invertible, we extend $f$ to $\D_-$ by using a left canonical factorization as follows:
\begin{equation*}
	f^e(z) = f^e_+(\bar{z}^{-1})f^e_-(z),
\end{equation*}
where $z \in \D_-$.
This $f^e$ is also a continuous nonsingular matrix-valued function and is independent of the choice of left canonical factorization, which can be shown as in \cite[Lemma~3.1]{Hayashi5}.
Combining these results, the following lemma is obtained as in \cite[Lemma~3.3]{Hayashi5}.

\begin{lemma}
\label{extension1}
Let $Y$ be a topological space.
Let $f \colon \T \times Y \to GL(n,\C)$ be a continuous map such that for each $y$ in $Y$, $f(y)$ is a rational matrix function on $\T$ with trivial left and right partial indices.
Through a Wiener--Hopf factorization, there canonically associates a continuous map $f^e \colon \SSS^2 \times Y \to GL(n,\C)$ that extends $f$.
\end{lemma}
Note that if $f$ further takes values in Hermitian matrices, its extension $f^e$ is also a Hermitian matrix-valued function \cite[Lemma~3.5]{Hayashi5}.

We next consider these matrix-valued functions that preserve some symmetry.
Let $\Pi$ and $I$ be $n$-by-$n$ Hermitian unitary matrices.

\begin{lemma}
\label{extension2}
Let $f \colon \T \to GL(n,\C)$ be a rational matrix function with poles off $\T$ of trivial left and right partial indices.
If $f(z)$ anticommutes with $\Pi$ for each $z$ in $\T$, then $f^e(z)$ anticommutes with $\Pi$ for each $z$ in $\SSS^2$.
\end{lemma}
\begin{proof}
By assumption, there exists a right canonical factorization $f = f_- f_+$.
Since $f$ anticommutes with $\Pi$, we have,
\begin{equation*}
	f(z) = -\Pi f(z) \Pi^* = - \Pi f_-(z) \Pi^* \cdot \Pi f_+(z) \Pi^*,
\end{equation*}
where $z \in \T$.
This also is a right canonical factorization since the left and right factors are extended onto $\D_-$ and $\D_+$ by $-\Pi f^e_-(z)\Pi^*$ and $\Pi f^e_+(z) \Pi^*$, respectively.
Since the extension $f^e$ onto $\D_+$ is independent of the choice of right canonical factorization, the following holds for $z$ in $\D_+$:
\begin{equation*}
	f^e(z) = -\Pi f^e_-(\bar{z}^{-1})\Pi^* \cdot \Pi f^e_+(z) \Pi^* = -\Pi f^e(z) \Pi^*.
\end{equation*}
The relation on $\D_-$ follows similarly using left canonical factorization.
\end{proof}

In what follows, we write $\tT$ for the unit circle $\T$ equipped with the involution given by the complex conjugation.
Its $n$-fold product is denoted as $\tT^n$.
Let $\tS^2$ be the Riemann sphere $\SSS^2 = \C \cup \{ \infty \}$ equipped with the involution given by $z \mapsto z^{-1}$, where we set $0^{-1} = \infty$ and $\infty^{-1}=0$.
Note that $\tT$ is a $\Z/2$-subspace of $\tS^2$.

\begin{lemma}
\label{extension3}
Let $Y$ be a $\Z/2$-space. We write $\tau$ for the involution on $Y$.
Let $f \colon \T \times Y \to GL(n,\C)$ be a continuous map such that, for each $y$ in $Y$, $f(y)$ is a rational matrix function on $\T$ with trivial left and right partial indices.
If $f \colon \tT \times Y \to (GL(n,\C), \Ad_I)$ is a $\Z/2$-map; that is, $f$ satisfies
\begin{equation}
\label{lemsym}
	I f(z,y) I^* = f(\bar{z}, \tau(y)),
\end{equation}
for any $(z,y)$ in $\tT \times Y$,
then the extension $f^e$ in Lemma~\ref{extension1} is a $\Z/2$-map $f^e \colon \tS^2 \times Y \to (GL(n,\C), \Ad_I)$; that is, $f^e$
satisfies
\begin{equation*}
	I f^e(z,y) I^* = f^e(z^{-1}, \tau(y)),
\end{equation*}
for any $(z,y)$ in $\tS^2 \times Y$.
\end{lemma}

\begin{proof}
Let $y$ be an element in $Y$.
By assumption, the matrix-valued function $f(y)$ on $\T$ has a right canonical factorization $f(z,y) = f_-(z,y)f_+(z,y)$.
Combined with the symmetry (\ref{lemsym}), we also have the following factorization of $f(\tau(y))$:
\begin{equation*}
	f(z,\tau(y)) = I f(\bar{z}, y)I^*
		= I f_-(z^{-1}, y)I^* \cdot I f_+(z^{-1}, y) I^*,
\end{equation*}
where $z \in \T$.
This is a {\em left} canonical factorization since the left and right factors are extended onto $\D_+$ and $\D_-$ by $I f_-^e(z^{-1}, y)I^*$ and $I f_+^e(z^{-1}, y)I^*$, respectively.
Therefore, for $z$ in $\D_-$, our extension satisfies
\begin{equation*}
	f^e(z,\tau(y)) = I f^e_-(\bar{z}, y)I^* \cdot I f^e_+(z^{-1}, y) I^*
		= I f^e(z^{-1}, y) I^*,
\end{equation*}
which completes the proof.
\end{proof}

\begin{remark}
\label{reminv}
In the following, $\Pi$ and $I$ appears as chiral and inversion symmetry operators, respectively. 
Inversion symmetry relates two Toeplitz operators on different discrete half lines.
Correspondingly, as we can see in the proof, the assumptions of trivial left and right partial indices in Lemma~\ref{extension3} are redundant.
Under the symmetry (\ref{lemsym}), if $f(y)$ has trivial right partial indices (equivalently, the associated Toeplitz operator $T^+_{f(y)}$ is invertible), then $f(\tau(y))$ has trivial left partial indices (the Toeplitz operator $T^-_{f(\tau(y))}$ is invertible), and vice versa.
\end{remark}

\subsection{$K^0_{\Z/2}$- and $K^1_\pm$-groups}
\label{Sect2.3}
In this subsection, we collect the necessary results concerning $K$-theory.
For the computation of $\Z/2$-equivariant $K$-groups in this paper, we define them as the Grothendieck group of the isomorphism classes of finite-rank $\Z/2$-equivariant complex vector bundles \cite{At67,Se68}.
We also introduce $K_{\Z/2}$-theory and $K_\pm$-theory in Karoubi's formulation using pairs of gradations \cite{Karoubi68, Karoubi78}.
Regarding their definitions, we employ the finite-dimensional Karoubi formulation presented in \cite{Gomi23} for this context following \cite{FM13}.
We also introduce some explicit isomorphism between the $K^1_\pm$- and $K^0_{\Z/2}$-groups which will be used in Sect.~\ref{Sect4}.
For the details of these $K$-theory groups, we refer the reader to \cite{Karoubi68,Se68,Karoubi78,AH04,Gomi15,Gomi23}.

Let $Y$ be a compact Hausdorff $\Z/2$-space.
We consider a {\em triple} $(E, \eta_0,\eta_1)$ on $Y$ consisting of a finite-rank $\Z/2$-equivariant Hermitian vector bundle $E$ on $Y$ and self-adjoint involutions $\eta_0$ and $\eta_1$ on $E$ that commute with the $\Z/2$-action.
These involutions are called {\em gradations}.
An isomorphism of triples $f \colon (E, \eta_0,\eta_1) \to (E', \eta'_0,\eta'_1)$ is given by an isomorphism of vector bundles $f \colon E \to E'$ that provides isomorphisms $f \colon (E, \eta_i) \to (E', \eta'_i)$ of graded equivariant vector bundles for $i=0,1$.
The direct sum of these triples is given by $(E, \eta_0, \eta_1) \oplus (E', \eta'_0, \eta'_1) = (E \oplus E', \eta_0 \oplus \eta'_0, \eta_1 \oplus \eta'_1)$.
Let $M_{\Z/2}(Y)$ be the monoid of the isomorphism classes of the triples and $Z_{\Z/2}(Y)$ be its submonoid consisting of the classes of the triples with the form $(E, \eta_0, \eta_1)$ such that $\eta_0$ is homotopic to $\eta_1$ as gradations.
The quotient monoid
\begin{equation*}
	K^0_{\Z/2}(Y)_{\Kar} = M_{\Z/2}(Y) / Z_{\Z/2}(Y),
\end{equation*}
is an abelian group.
We simply write $[E, \eta_0, \eta_1]$ for the class of the triple $(E, \eta_0, \eta_1)$ in $K^0_{\Z/2}(Y)_{\Kar}$.
Let $Y' \subset Y$ be a closed $\Z/2$-subspace.
The relative $K_{\Z/2}$-group $K^0_{\Z/2}(Y, Y')_\Kar$ is defined similarly by using triples $(E, \eta_0, \eta_1)$ that further satisfy the relation $\eta_0|_{Y'} = \eta_1|_{Y'}$.
Homotopies between $\eta_0$ and $\eta_1$ for taking the quotient are taken to be constant on $Y'$.
A relation with the group $K^0_{\Z/2}(Y)$ defined as the Grothendieck group of the isomorphic classes of finite-rank $\Z/2$-equivariant complex vector bundles is given by the isomorphism
\begin{equation}
\label{eq4}
	K^0_{\Z/2}(Y)_{\Kar} \to K^0_{\Z/2}(Y),
\end{equation}
which maps $[E, \eta_0, \eta_1]$ to $\Ker(1 + \eta_0) - \Ker(1 + \eta_1)$.

\begin{remark}
\label{rem2.5}
For the isomorphism (\ref{eq4}), see \cite[Chapter~III, Proposition~5.7]{Karoubi78} and \cite[Proposition~4.21]{Gomi23}, where
an isomorphism is given by assigning $\Ker(1-\eta)$ for the gradation $\eta$.
Our $\Ker(1+\eta)$ is chosen for consistency with the study of topological insulators.
For a bulk-gapped Hamiltonian $H$, the element $\eta = H|H|^{-1}$ is a gradation, and $\Ker(1+\eta)$ corresponds to the valence band.
\end{remark}

We next introduce $K_\pm$-groups.
Following \cite{Gomi23}, we define $K^1_{\pm}$-groups by using a finite-rank {\em $c$-twisted vector bundle with $Cl_{0,1}$-action}, where $c$ is the identity map $\id_{\Z/2}$ on $\Z/2$, which consists of the following data $(E, \eta, \rho, \gamma)$.
\begin{itemize}
\item $E$ is a finite-rank $\Z/2$-equivariant Hermitian vector bundle on $Y$. We write $\rho$ for the $\Z/2$-action\footnote{In what follows, since the $\Z/2$-action $\rho$ is provided by the involution $\rho(1)$ on $E$, where $1$ is the generator of the group $\Z/2$, we may simply write $\rho$ for the involution $\rho(1)$.} on $E$.
\item $\eta$ is a gradation on $E$.
\item $\gamma$ provides an action of the Clifford algebra $Cl_{0,1}$ on $E$; that is, $\gamma$ is a unitary homomorphism on $E$ (lifting the identity on $Y$) satisfying $\gamma^2=1$.
\end{itemize}
The above satisfies $\eta \rho = - \rho \eta$, $\gamma \eta = - \eta \gamma$ and $\gamma \rho = - \rho \gamma$.
We consider the triple $(E, \eta_0, \eta_1)$ such that $E = (E, \rho, \gamma)$, equipped with the gradation $\eta_i$, is a $c$-twisted vector bundle for $i=0,1$.
An isomorphism of triples $f \colon (E, \eta_0, \eta_1) \to (E', \eta'_0, \eta'_1)$ is given by an isomorphism of vector bundles $f \colon E \to E'$ that provides isomorphisms $f \colon (E, \eta_i, \rho, \gamma) \to (E', \eta'_i, \rho', \gamma')$ of the $c$-twisted vector bundles for each $i=0,1$.
Let $M_\pm(Y)$ be the monoid consisting of the isomorphism classes of such triples whose binary operation is provided by the direct sum, and let $Z_\pm(Y)$ be its submonoid consisting of the triples of homotopic gradations.
The quotient monoid
\begin{equation*}
	K^1_{\pm}(Y)_{\Kar} = M_{\pm}(Y) / Z_{\pm}(Y)
\end{equation*}
is an abelian group.
In what follows, we remove the subscript $\mathrm{Kar}$ from our notation for the $K$-groups in Karoubi's formulation, although the representatives of the $K$-groups will be clear from their context.
Note that we have the following isomorphisms \cite{Gomi15}:
\begin{equation}
\label{isom}
	K^n_{\Z/2}(Y \times \tT) \cong K^n_{\Z/2}(Y) \oplus K^{n-1}_{\pm}(Y), \ \ \
	K^n_{\pm}(Y \times \tT) \cong K^n_{\pm}(Y) \oplus K^{n-1}_{\Z/2}(Y).
\end{equation}

\begin{remark}
\label{interval}
Let $J$ be the interval $[-\frac{\pi}{2},\frac{\pi}{2}]$ equipped with the involution given by $\theta \mapsto -\theta$.
In \cite{AH04}, the group $K^n_{\pm}(Y)$ is defined as the $\Z/2$-equivariant $K$-theory group $K^{n+1}_{\Z/2}(Y \times J, Y \times \partial J)$.
For our applications to topological insulators, Karoubi's formulation enables us to provide an element of a $K^1_\pm$-group from pairs of Hamiltonians that preserve some symmetry \cite{Kit09, Thi16, Gomi23}.
\end{remark}

Finally, we briefly mention the Thom isomorphism between $K^1_{\pm}$- and $K^0_{\Z/2}$-groups \cite{AH04,Gomi15}.
This is an equivariant analog of \cite[Chapter~III, Theorem~5.10]{Karoubi78}, and the result will be well known to experts \cite{Karoubi68}.
Since the specific form of the isomorphism is required in Sect.~\ref{Sect4}, the result is briefly contained.

Let $Y$ be a compact Hausdorff $\Z/2$-space. We write its involution as $\tau$.
By using the interval $J$ equipped with the nontrivial $\Z/2$-action in Remark~\ref{interval}, let
\begin{equation}\label{t}
	t_Y \colon K^1_\pm(Y) \to K^0_{\Z/2}(Y \times J, Y \times \partial J)
\end{equation}
be the map constructed as follows.
Let $[E, \eta_0, \eta_1]$ be an element in $K^1_\pm(Y)$.
For $y \in Y$, $\theta \in J$ and $i=0,1$, let
\begin{equation*}
	\zeta_i(y,\theta) = \eta_i(y) \cos \theta - \gamma \sin \theta,
\end{equation*}
which satisfies $\zeta_i^2=1$, and the relation $\zeta_0|_{Y \times \partial J} = \zeta_1|_{Y \times \partial J}$ holds.
Let $\pi \colon Y \times J \to Y$ be the projection.
For the bundle $\pi^*E$, we consider the $\Z/2$-action given by $\chi = \sqrt{-1} \gamma \rho$.
Each $\zeta_i$ is a gradation on $\pi^*E$ and satisfies the following $\Z/2$-symmetry:
\begin{align*}
	\chi \zeta_i(y,\theta) \chi^* &= \chi \eta_i(y) \chi^* \cos \theta - \chi \gamma \chi^* \sin \theta\\
		&= \eta_i(\tau(y)) \cos \theta + \gamma \sin \theta \\
		&= \eta_i(\tau(y)) \cos (-\theta) - \gamma \sin (-\theta) \\
		&= \zeta_i(\tau(y), -\theta).
\end{align*}
We set $t_Y([E, \eta_0, \eta_1]) = [\pi^*E, \zeta_0, \zeta_1]$, which provides a natural group homomorphism.

In what follows, we may simply write $R$ for the representation ring $R(\Z/2) = \Z[t]/(t^2-1)$.
Let us consider the composite of the following maps:
\begin{equation*}
	K^1_\pm(Y) \overset{t_Y}{\longrightarrow} K^0_{\Z/2}(Y \times J, Y \times \partial J)
		\to K^0_{\Z/2}(Y \times J) \cong K^0_{\Z/2}(Y),
\end{equation*}
where the second map is induced by the inclusion $(Y \times J, \varnothing) \to (Y \times J, Y \times \partial J)$.
The composite of the above maps is given by mapping an $\id_{\Z/2}$-twisted vector bundle with a $Cl_{0,1}$-action $(E, \eta, \rho, \gamma)$ to the $\Z/2$-equivariant graded vector bundle $(E, \eta, \chi)$, where $\chi = \sqrt{-1} \gamma \rho$.
If there is a fixed point $y_0$ of the involution on $Y$, combined with the restriction homomorphism
$K^0_{\Z/2}(Y) \to K^0_{\Z/2}(\{y_0\}) = R$,
we obtain the following homomorphism:
\begin{equation}
\label{restriction}
	K^1_\pm(Y) \to R.
\end{equation}
We simply call the map (\ref{restriction}) the {\em restriction homomorphism} in this paper.
By using the identification~(\ref{eq4}), the above restriction homomorphism simply counts the number of irreducible representations of $\Z/2$ on the fixed point of the involution in $Y$, but with a modified $\Z/2$-action; that is, through our restriction homomorphism, we replace the $\Z/2$-action $\rho$ on $E$ with $\chi$.

\begin{remark}
\label{action}
The above $\Z/2$-action $\chi$ commutes with the gradation $\eta$ and anticommutes with the Clifford action $\gamma$.
Once the $\id_{\Z/2}$-twisted vector bundle with the $Cl_{0,1}$-action $(E, \eta, \rho, \gamma)$ is fixed, the correspondence between $\rho$ and $\chi$ is one-to-one.
In this sense, we can define the group $K^1_\pm(Y)$ from quadruples $(E, \eta, \chi, \gamma)$, and within this formulation, the map in (\ref{restriction}) is simply the restriction homomorphism.
In this paper, we define the $K^1_\pm$-group according to \cite{Gomi15} and mention the above modification of the $\Z/2$-action within our restriction homomorphism (\ref{restriction}).
In our applications in Sect.~\ref{Sect4}, we first identify quadruples $(E, \eta, \chi, \gamma)$ from our setup and then adjust the group action $\chi$ to $\rho$ to define an element in the $K^1_\pm$-group.
\end{remark}

\begin{example}
\label{expt}
Let
$\epsilon =
\begin{pmatrix}
0 & 1\\
1 & 0
\end{pmatrix}$,
$\rho =
\begin{pmatrix}
0 & -\sqrt{-1}\\
\sqrt{-1} & 0
\end{pmatrix}$
and
$\gamma =
\begin{pmatrix}
1 & 0\\
0 & -1
\end{pmatrix}$.
The quadruple $(\C^2, \epsilon, \rho, \gamma)$ is an $\id_{\Z/2}$-twisted vector bundle with a $Cl_{0,1}$-action on the point.
In this case, we have $\chi = \sqrt{-1}\gamma \rho = \epsilon$.
Note that $K^0_{\Z/2}(J, \partial J) \cong \Z$ and that $K^1_\pm(\pt) \cong K^2(\pt) \cong \Z$, where $\pt$ denotes the one-point set equipped with the trivial $\Z/2$-action.
The element $[\C^2, -\epsilon, \epsilon] \in K^1_{\pm}(\pt)$ maps to $1-t \in R$ by our restriction homomorphism.
This computation indicates that the $K$-class $[\C^2, -\epsilon, \epsilon]$ generates the group $K^1_{\pm}(\pt)$, and its image under the map $t_{\pt}$ generates the group $K^0_{\Z/2}(J, \partial J)$.
Therefore, the map $t_{\pt}$ is an isomorphism.
\end{example}

The following makes part of the isomorphisms in (\ref{isom}) explicit.
\begin{proposition}
\label{mapt}
The map $t_Y$ in (\ref{t}) is a group isomorphism.
\end{proposition}
\begin{proof}
We follow the discussions in \cite{AH04}.
The classifying space of $K^0_{\Z/2}$-groups are provided in \cite{Mat71b,KP74}.
Following \cite{AS69}, the classifying space for $K^1_{\pm}$-groups can be constructed by starting from the classifying space for $K^0_{\pm}$-groups in \cite{AH04}.
As in \cite[Chapter~III]{Karoubi78}, the map $t$ in (\ref{t}) induces a $\Z/2$-map between these spaces.
Therefore, by the equivariant Whitehead theorem \cite{Bre67,Mat71a}, it is sufficient to show that $t_Y$ is an isomorphism for spaces $Y$ with trivial $\Z/2$-action and $\Z/2 \times Y$ (in other words, forget the $\Z/2$-action on $Y$).

If we forget the $\Z/2$-action on $Y$, then the map
$t_Y \colon K^{1}(Y) \to K^0(Y \times J, Y \times \partial J)$ is shown to be the isomorphism in \cite[Chapter~III, Theorem 5.10]{Karoubi78}.
When the $\Z/2$-action on $Y$ is trivial (that is, $\tau = \id_Y$), we have
\begin{equation*}
K^1_\pm(Y)  \cong K^0(Y) \otimes K^1_\pm(\pt), \ \
K^0_{\Z/2}(Y \times J, Y \times \partial J) \cong K^0(Y) \otimes K^0_{\Z/2}(J, \partial J).
\end{equation*}
Through these isomorphisms, the map $t_Y$ corresponds to $1 \otimes t_{\pt}$.
By Example~\ref{expt}, the map $t_{\pt}$ is an isomorphism.
Therefore, the map $t_Y$ is an isomorphism, which completes the proof.
\end{proof}

\section{Bulk-hinge correspondence}
\label{Sect3}
In this section, we consider three-dimensional systems that preserve inversion symmetry and discuss the bulk-hinge correspondence for second-order topological insulators.

\subsection{Setup}
\label{Sect3.1}
Let $N$ be a positive integer
and $H \colon \T^3 \to M(N,\C)$ be a continuous and Hermitian matrix-valued function.
We write an element of $\tT^3$ as $(z,w,t)$, where $z$, $w$ and $t$ are elements of $\T$.
We assume that for each $t$ in $\T$, the map $H(t) \colon \T^2 \to M(N,\C)$ is a two-variable rational matrix function concerning the variables $z$ and $w$.
Our $H$ is a continuous map to $M(N, \C)$ and there are no poles on $\T^3$.
We may simply write $H$ for the multiplication operator $M_H$ on the Hilbert space $l^2(\Z^3, \C^N)$ generated by the matrix-valued function $H$.
Translation-invariant single-particle Hamiltonians on the lattice $\Z^3$ with finitely many hopping terms correspond to three-variable Laurent polynomial matrix-valued functions on $\T^3$ and are therefore within this class.
From this viewpoint, we call $H$ the bulk Hamiltonian.
Let $I$ be an $N$-by-$N$ Hermitian unitary matrix.
We write an element $\varphi$ in $l^2(\Z^3, \C^N)$ as a family $\{ \varphi_\mathfrak{n} \}_{\mathfrak{n} \in \Z^3}$ of vectors $\varphi_{\mathfrak{n}}$ in $\C^N$ satisfying $\sum_{\mathfrak{n} \in \Z^3}||\varphi_{\mathfrak{n}}||^2 < \infty$.
By abuse of notation, we also write $I$ for the operator on $l^2(\Z^3, \C^N)$ defined by $(I \varphi)_{\mathfrak{n}} = I \varphi_{-\mathfrak{n}}$.
In what follows, we consider Hamiltonians that preserve inversion symmetry.
Correspondingly, we assume that our Hamiltonian $H$ satisfies the relation $I H I^* = H$ as operators on $l^2(\Z^3, \C^N)$.
In terms of matrix-valued functions on $\T^3$, the inversion symmetry is described as follows:
\begin{equation}
\label{inversionHam}
	I H(z, w, t) I^* = H(\bar{z}, \bar{w}, \bar{t}),
\end{equation}
where $(z,w,t) \in \T^3$.
For $i=1,2,3,4$ and $j=a,b,c,d$, let $H^i(t) = T^i_{H(t)}$ and $H^j(t) = T^j_{H(t)}$ for each $t$ in $\T$, that is, the half-plane and quarter-plane Toeplitz operators of symbol $H(t)$, respectively.
These operators are our models for four surfaces and four hinges of rectangular prism-shaped systems (see Fig.~\ref{Fig1}).

We next assume that these inversion-symmetric Hamiltonians have a spectral gap at zero energy\footnote{In what follows, the word {\em gapped} always means {\em gapped at zero energy}. The following Theorem~\ref{theorem3D} also holds when the spectrum of $H$ is contained in $\R_{>0}$ or $\R_{<0}$, but the conclusion is trivial.}.
Correspondingly, we assume that zero is not contained in the spectrum of $H$; in other words, our matrix-valued function $H$ takes values in the general linear group $GL(N,\C)$.
Under this assumption, we introduce a gapped bulk topological invariant, which is called the symmetry-based indicator.
For such a gapped bulk Hamiltonian, its {\em Bloch bundle} is a vector bundle $\Bloch$ on $\T^3$ whose fiber at a point $(z,w,t)$ in $\T^3$ is given as follows:
\begin{equation}
\label{Bloch}
	 (\Bloch)_{(z,w,t)} = \bigcup_{\mu < 0} \Ker (H(z,w,t) - \mu).
\end{equation}
Owing to the presence of inversion symmetry, the Bloch bundle is a $\Z/2$-equivariant vector bundle on $\tT^3$.
Let $E$ be a finite-rank $\Z/2$-equivariant complex vector bundle on $\tT^3$, and let $\Gamma^3_i$ ($i= 1,\ldots, 8$) be the eight fixed points of the involution on $\tT^3$.
For each $\Gamma^3_i$, there is a restriction homomorphism from $K^0_{\Z/2}(\tT^3)$ to $K^0_{\Z/2}(\{\Gamma^3_i\})$.
Through this map, the $K$-class of the bundle $E$ in $K^0_{\Z/2}(\tT^3)$ is mapped to an element in $R = \Z[t]/(t^2-1)$.
We denote this element as $n_+(\Gamma^3_i)1 + n_-(\Gamma^3_i)t$, where $n_\pm(\Gamma^3_i)$ are nonnegative integers.

\begin{remark}
\label{remark3.1}
If we take the Bloch bundle $\Bloch$ for the bundle $E$, the nonnegative integers $n_{\pm}(\Gamma^3_i)$ count the number of occupied states with even (+) and odd (-) parities with respect to the inversion symmetry operator $I$ at the time-reversal invariant momentum $\Gamma^3_i$.
\end{remark}

Following \cite{Po17,Khalaf18a,TTM20a}, we introduce a topological invariant for inversion-symmetric bulk-gapped Hamiltonians.
\begin{definition}[Symmetry-based indicator]
\label{defind3D}
Let $E$ be a finite-rank $\Z/2$-equivariant complex vector bundle on $\tT^3$.
We define
\begin{equation*}
	\mu_{3D}(E) = - \sum_{i=1}^8 n_-(\Gamma^3_i) \ \mod 4,
 \end{equation*}
which induces the group homomorphism $\mu_{3D} \colon K^0_{\Z/2}(\tT^3) \to \Z/4$.
\end{definition}
In what follows, we may simply write $\mu_{3D}(H)$ for $\mu_{3D}(\Bloch)$.

\begin{lemma}
\label{symindeven}
The map
$\mu_{3D} \colon K^0_{\Z/2}(\tT^3) \to \Z/4$
takes values of $0$ and $2$.
\end{lemma}

\begin{proof}
Once the generators of the group $K^0_{\Z/2}(\tT^3)$ are specified, it is sufficient to compute their symmetry-based indicators,
and the result follows from \cite[Figure~1]{Gomi15} (also included as Table~\ref{table3} in this paper).
\end{proof}
By Lemma~\ref{symindeven}, we obtain the following group homomorphism:
\begin{equation}
\label{1/2}
\frac{1}{2}\mu_{3D} \colon K^0_{\Z/2}(\tT^3) \to \Z/2.
\end{equation}

To discuss topological hinge states, we additionally assume the following surface gapped condition.

\begin{assumption}[Spectral gap condition]
\label{assumption3D}
For each $i=1,2,3,4$ and $t$ in $\T$, the half-plane Toeplitz operators $H^i(t)$ are invertible.
\end{assumption}

Assumption~\ref{assumption3D} is satisfied if a spectral gap is present for both the bulk Hamiltonian and our four surface models\footnote{
It follows from Assumption~\ref{assumption3D} that the matrix-valued function $H$ takes values in invertible matrices.
We begin with the bulk gapped condition and then address the surface gapped condition for Remark~\ref{rem4.6}.}.
Furthermore, under Assumption~\ref{assumption3D}, our models for four hinges, $\{ H^j(t) \}_{t \in \T}$ ($j=a,b,c,d$), form a family of self-adjoint Fredholm operators \cite{DH71,Pa90}.
We denote their spectral flow as $\sf_j(H) \in \Z$, where $j=a,b,c,d$.
The spectral flow accounts for the number of chiral hinge states located at each of the four hinges \cite{Hayashi2}.
The inversion symmetry interchanges hinges ($a$ with $c$ and $b$ with $d$), reversing the directions of the hinges.
For example, $\{ H^a(t) \}_{t \in \T}$ is identified with $\{ H^c(\bar{t}) \}_{t \in \T}$.
This symmetry leads to the following relations:
\begin{equation}
\label{inversionsf}
	\sf_a(H) = - \sf_c(H), \ \ \sf_b(H) = - \sf_d(H).
\end{equation}
Following \cite{TTM20a}, we formulate the bulk-hinge correspondence as follows:

\begin{theorem}[Bulk-hinge correspondence]
\label{theorem3D}
Under Assumption~\ref{assumption3D}, the following equality holds:
\begin{equation*}
	\frac{1}{2}\mu_{3D}(H) = \sf_a(H) + \sf_b(H) \ \mod 2.
\end{equation*}
\end{theorem}

By Theorem~\ref{theorem3D}, when a bulk-surface gapped inversion-symmetric Hamiltonian satisfies $\frac{1}{2}\mu_{3D}(H) = 1$, then at least $\sf_a(H) = - \sf_c(H)$ or $\sf_b(H) = -\sf_d(H)$ is nontrivial, and topological hinge states appear.

For the proof of Theorem~\ref{theorem3D}, the idea is to use the extension of the bulk Hamiltonian $H$ as in \cite{Hayashi5} under our four surface gapped conditions (Assumption~\ref{assumption3D}).
The key diagram is shown in Fig.~\ref{diagram3D}, where the discussion is reduced to the computation of (equivariant) topological $K$-theory groups and the maps between them.

\subsection{Key diagram}
\label{Sect3.2}
For our matrix-valued function $H$, we first consider $(w,t)$ in $\T^2$ as a parameter and fix it.
The continuous map $H(w,t) \colon \T \to GL(N,\C)$ is then a (single-variable) rational matrix function.
By Assumption~\ref{assumption3D}, the half-plane Toeplitz operators $H^1(t)$ and $H^3(t)$ are invertible for any $t$ in $\T$; therefore, the Toeplitz operators $T^+_{H(w,t)}$ and $T^-_{H(w,t)}$ are invertible for any $(w,t)$ in $\T^2$.
Therefore, $H(w,t)$ has trivial left and right partial indices.
By Lemma~\ref{extension1}, a continuous map $H^e \colon \SSS^2 \times \T^2 \to GL(N,\C)$ that extends $H$ on $\T^3$ is canonically associated.
Since our Hamiltonian $H$ preserves inversion symmetry (\ref{inversionHam}), by Lemma~\ref{extension3}, the extension $H^e$ satisfies the relation $IH^e(z,w,t)I^* = H^e(z^{-1}, \bar{w}, \bar{t})$ for any $(z,w,t)$ in $\SSS^2 \times \T^2$.
Similarly, by the invertibility of $H^2(t)$ and $H^4(t)$ in Assumption~\ref{assumption3D}, the extension $H^e \colon \T \times \SSS^2 \times \T \to GL(N,\C)$ of $H$ that satisfies $IH^e(z,w,t)I^* = H^e(\bar{z}, w^{-1}, \bar{t})$ for any $(z,w,t)$ in $\T \times \SSS^2 \times \T$ is canonically associated.
Since $H$ takes values in Hermitian matrices, its extension $H^e$ is also a Hermitian matrix-valued function.
Let us introduce the following space:
\begin{equation}
\label{X}
	X = (\D_+ \times \T) \underset{\T^2}{\cup}
		(\D_- \times \T) \underset{\T^2}{\cup}
			(\T \times \D_+)	 \underset{\T^2}{\cup}
				(\T \times \D_-).
\end{equation}
Note that the space $X$ is a subspace of $\SSS^2 \times \SSS^2$.
By using the variables $z$ and $w$ in $\SSS^2$, we write $(z,w)$ for the elements in $X$.
We equip $X$ with the involution given by $(z,w) \mapsto (z^{-1}, w^{-1})$.
Note that the torus $\tT^2$ is a $\Z/2$-subspace of $X$.
Summarizing the above discussion using the $\Z/2$ space $X$, starting from the inversion-symmetric Hamiltonian $H$ satisfying our Assumption~\ref{assumption3D}, there is a canonically associated nonsingular matrix-valued continuous map $H^E$ that extends $H$:
\begin{equation*}
	H^E \colon X \times \tT \to GL(N,\C),
\end{equation*}
which takes values in Hermitian matrices.
This extended Hamiltonian $H^E$ satisfies the following $\Z/2$-symmetry:
\begin{equation*}
	I H^E(z, w, t) I^* = H^E(z^{-1}, w^{-1}, \bar{t}),
\end{equation*}
where $(z,w,t) \in X \times \tT$.
Associated with this extended Hamiltonian $H^E$, there is a finite-rank $\Z/2$-equivariant complex vector bundle $\eBloch$ on $X \times \tT$, which is defined as in (\ref{Bloch}).
As $H^E$ is an extension of $H$, this bundle $\eBloch$ is an extension of the Bloch bundle $\Bloch$ on $\tT^3$ to $X \times \tT$.
This bundle $\eBloch$ provides an element of the group $K^0_{\Z/2}(X \times \tT)$.

We consider the following diagram in Fig.~\ref{diagram3D} containing the group $K^0_{\Z/2}(X \times \tT)$.
The maps $f_i$ are forgetful maps, and $\phi_1$, $\varphi_1$ and $\psi_1$ are restriction homomorphisms induced by the inclusions $\tT^3 \subset \tS^2 \times \tT \times \tT \subset X \times \tT$.
To introduce the bottom-left vertical map $c_2$, we consider the following four spaces:
\begin{equation}
\label{Spm}
	\SSS_{\pm\pm}^3 = (\D_{\pm} \times \T) \underset{\T^2}{\cup} (\T \times \D_{\pm}),
\end{equation}
where double signs correspond.
Note that each of the spaces $\SSS_{\pm\pm}^3$ is a subspace of $X$ and is homeomorphic to the three-sphere.
The map $c_2$ is the composite of the map provided by the direct sum of the four restriction homomorphisms,
\begin{equation*}
	K^0(X \times \T) \to K^0(\SSS^3_{++} \times \T) \oplus K^0(\SSS^3_{-+} \times \T) \oplus K^0(\SSS^3_{--} \times \T) \oplus K^0(\SSS^3_{+-} \times \T),
\end{equation*}
and the map given by taking the second Chern number for each group $K^0(\SSS_{\pm\pm} \times \T)$.

\begin{figure}[htbp]
\centering
\begin{tabular}{c}
\xymatrix{
K^0_{\Z/2}(X \times \tT) \ar[r]^{\phi_1 \hspace{2mm}} \ar[d]^{f_1}
	& K^0_{\Z/2}(\tS^2 \times \tT^2) \ar[r]^{\hspace{5mm} \psi_1} \ar[d]^{f_2}
		& K^0_{\Z/2}(\tT^3) \ar[r]^{\hspace{5mm} \frac{1}{2}\mu_{3D}}
			& \Z/2\\
K^0(X \times \T) \ar[r]^{\varphi_1 \hspace{2mm}} \ar[d]^{c_2}
	& K^0(\SSS^2 \times \T^2) & & \\
\Z^4 & & &}
\end{tabular}
\caption{Key diagram for Theorem~\ref{theorem3D}}
\label{diagram3D}
\end{figure}

Since $\psi_1 \circ \phi_1$ is the restriction homomorphism and the bundle $\eBloch$ is the extension of the bundle $\Bloch$, the image of the $K$-class of $\eBloch$ in $K^0_{\Z/2}(X \times \tT)$ via the horizontal maps in Fig.~\ref{diagram3D} is as follows:
\begin{equation}
\label{mu3D}
	\frac{1}{2}\mu_{3D} \circ \psi_1 \circ \phi_1 (\eBloch) = \frac{1}{2}\mu_{3D}(\Bloch) = \frac{1}{2}\mu_{3D}(H).
\end{equation}
For the vertical maps, the image under the map $c_2 \circ f_1$ is a four-tuple of spectral flows of the families of quarter-plane Toeplitz operators \cite{Hayashi5}, that is,
\begin{equation}
\label{sf}
	c_2 \circ f_1 (\eBloch) = (\sf_a(H), \sf_b(H), \sf_c(H), \sf_d(H)),
\end{equation}
holds.
In what follows, we compute the maps in Fig.~\ref{diagram3D} in Sects.~\ref{Sect3.3} and \ref{Sect3.4} and discuss their relation in Sect.~\ref{Sect3.5}, which provides the proof of Theorem~\ref{theorem3D}.

\subsection{Computation of top horizontal maps}
\label{Sect3.3}
In this subsection, we focus on the top horizontal maps shown in Fig.~\ref{diagram3D}, which account for bulk topological invariants (\ref{mu3D}).
We compute equivariant $K$-groups, their generators, and the maps between them.
The necessary calculations were conducted in \cite{Gomi15}, which we mainly follow.
We collect the results in Sect.~\ref{Sect3.3.1} to Sect.~\ref{Sect3.3.4}, and Table~\ref{table3} is taken from \cite{Gomi15}.

\subsubsection{$\Z/2$-equivariant vector bundles}
\label{Sect3.3.1}
We start with discussions of $\Z/2$-equivariant vector bundles.
We write $\C_0$ and $\C_1$ for the trivial and nontrivial one-dimensional complex representations of $\Z/2$, respectively.
For a $\Z/2$-space $Y$ and a complex representation $V$ of $\Z/2$, we may simply write $\underline{V}$ for the product $\Z/2$-equivariant vector bundle $Y \times V$ on $Y$ when the base space is clear from the context.

We identify the Riemann sphere $\SSS^2$ with the complex projective line $\CP^1$.
The involution on $\tS^2$ corresponds to the involution $\tau_{\CP^1}([z:w]) = [w : z]$ on $\CP^1$.
Let $T$ be the tautological bundle on $\CP^1$.
Its fiber $T_{[z:w]}$ consists of the complex line in $\C^2$ through $(z,w) \in \C^2$ and the origin,
and the involution $(v_1,v_2) \mapsto (-v_2, -v_1)$ on $\C^2$ induces the involution on $T$.
Equipped with this involution, $T$ is a $\Z/2$-equivariant complex line bundle on $\tS^2 = (\CP^1, \tau_{\CP^1})$.
We write $\tH$ for the dual bundle of $T$, which is the hyperplane bundle on $\CP^1$, and its first Chern number is $1$.
$\tH$ has a $\Z/2$-equivariant vector bundle structure induced by the involution on $T$.
In what follows, we take $-1$ as the base point of $\tT$.
Let
\begin{equation}
\label{q}
q \colon \tT \times \tT \to \tT \times \tT / \tT \vee \tT,
\end{equation}
be the quotient map.
We identify the quotient space with $\tS^2$ such that $q(1,1)$ is $1$ in $\tS^2$ and the other $\Z/2$-fixed points on the torus map to $-1$ in $\tS^2$.
Let $H = q^*\tH$, which is a $\Z/2$-equivariant line bundle on $\tT^2$.

\subsubsection{$K^0_{\Z/2}(\tT)$}
\label{Sect3.3.2}
Let $L = \T \times \C$ be the product line bundle on $\T$.
For the total space $L$, we consider the involution given by $(z,v) \mapsto (\bar{z}, -zv)$.
Then, $L \to \tT$ is a $\Z/2$-equivariant line bundle.
By \cite[Lemma~4.3]{Gomi15}, the group $K^0_{\Z/2}(\tT)$ is isomorphic to $\Z^3$, which is generated by $\uC_0$, $\uC_1$ and $L$.
We have the following ring homomorphism:
\begin{equation*}
F_1 \colon K^0_{\Z/2}(\tT) \to K^0(\T) \oplus R_{1} \oplus R_{-1},
\end{equation*}
which is given by $F_1(x) = (f(x); x|_1, x|_{-1})$, where $R_i = R$ for $i=\pm1$.
The map $F_1$ is injective.
By \cite[Lemma~4.3]{Gomi15}, the following relation\footnote{In equation (\ref{eq1}) and the following, we may omit $\otimes$ for brevity; for example, $\uC_1 L$ denotes $\uC_1 \otimes L$.} holds in $K^0_{\Z/2}(\tT)$:
\begin{equation}
\label{eq1}
	\uC_1(\uC_0 - L) = - (\uC_0 - L).
\end{equation}

\subsubsection{$K^0_{\Z/2}(\tT^2)$}
\label{Sect3.3.3}
Let $\pi_i \colon \tT^2 \to \tT$ be the $i$-th projection and $L_i = \pi_i^*L$, where $i=1,2$.
By \cite[Lemma~4.4]{Gomi15}, the group $K^0_{\Z/2}(\tT^2)$ is isomorphic to $\Z^6$, which is generated by
$\{ \uC_0, \uC_1, H, \uC_1H, L_1, L_2\}$.
Among the several relations in $K^0_{\Z/2}(\tT^2)$ presented in \cite[Lemma~4.4]{Gomi15}, we note the following:
\begin{equation}
\label{eq2}
	(\uC_0 - L_1)(\uC_1 - L_2) = (\uC_0 - \uC_1)(\uC_0 - H).
\end{equation}

In \cite[Lemma~4.4]{Gomi15}, the Mayer--Vietoris sequence is employed to compute the group $K^0_{\Z/2}(\tT^2)$.
We contain the following result for later use.
Let
\begin{equation}
\label{UV}
\begin{aligned}
	U &= \bigl\{ \exp (2\pi \sqrt{-1} \theta) \in \tT \ | \ \tfrac{1}{4} \leq \theta \leq \tfrac{3}{4} \bigl\},\\
	V &= \bigl\{ \exp (2\pi \sqrt{-1} \theta) \in \tT \ | -\tfrac{1}{4} \leq \theta \leq \tfrac{1}{4} \bigl\},
\end{aligned}
\end{equation}
be $\Z/2$-subspaces of $\tT$ satisfying $U \cup V = \tT$.
Let $\varphi_{\Z/2} \colon K^0_{\Z/2}(\tT^2) \to K^0_{\Z/2}(\tT \times U) \oplus K^0_{\Z/2}(\tT \times V)$ be the map in the Mayer--Vietoris sequence for $\{ \tT \times U,  \tT \times V \}$.
Note that $K^0_{\Z/2}(\tT \times U)$ and $K^0_{\Z/2}(\tT \times V)$ are isomorphic to $K^0_{\Z/2}(\tT)$.
The images of the generators of the group $K^0_{\Z/2}(\tT^2)$ under the map $\varphi_{\Z/2}$ are computed via composition with the injection $F_1 \oplus F_1$, and the result is as follows:
\begin{align*}
	\varphi_{\Z/2}(\uC_0) &= (\uC_0, \uC_0),&
	\varphi_{\Z/2}(\uC_1) &= (\uC_1, \uC_1),&
	\varphi_{\Z/2}(H) &= (\uC_0, L), \\
	\varphi_{\Z/2}(\uC_1 H) &= (\uC_1, \uC_0 + \uC_1 - L),&
	\varphi_{\Z/2}(L_1) &= (L, L),&
	\varphi_{\Z/2}(L_2) &= (\uC_0, \uC_1).
\end{align*}

\subsubsection{$K^0_{\Z/2}(\tT^3)$}
\label{Sect3.3.4}
Let $\pi_i \colon \tT^3 \to \tT$ be the $i$-th projection and $L_i = \pi_i^*L$ for $i=1,2,3$.
We also write $H_{ij} = (\pi_i, \pi_j)^*H$ for $1 \leq i < j \leq 3$.
By \cite[Lemma~4.5]{Gomi15}, the group $K^0_{\Z/2}(\tT^3)$ is isomorphic to $\Z^{12}$, whose generators are given as follows:
\begin{equation*}
\{ \uC_0, \uC_1, H_{12}, H_{23}, H_{13}, \uC_1H_{12}, \uC_1H_{23}, \uC_1H_{13}, L_1, L_2, L_3, H_{12}L_3 \}.
\end{equation*}
Let us consider the following ring homomorphism:
\begin{equation*}
F_3 : K^0_{\Z/2}(\tT^3) \longrightarrow K^0(\T^3) \oplus  \bigoplus_{(i, j, k) = (\pm 1, \pm 1, \pm 1)} R_{(i, j, k)},
\end{equation*}
given by $F_3(x) = (f(x), x|_{(i,j,k)})$, where $f \colon K^0_{\Z/2}(\tT^3) \to K^0(\T^3)$ is the forgetful map and $R_{(i, j, k)} = R$.
The map $F_3$ is injective, and the following holds in $K^0_{\Z/2}(\tT^3)$:
\begin{equation}
\label{eq3}
	(\uC_0 - H_{12})(\uC_0 - L_3) = (\uC_0 - H_{13})(\uC_0 - L_2) = (\uC_0 - H_{23})(\uC_0 - L_1),
\end{equation}
since their images under the map $F_3$ are $(0 ; 2(1-t), 0,0,0,0,0,0,0)$ in the notation in Table~\ref{table3}.

\begin{table}
\caption{The values of $F_3$ for $K^0_{\Z/2}(\tT^3)$ (from \cite[Fig.1]{Gomi15})}
\label{table3}
\begin{tabular}{c@{}|c|@{}c@{}c@{}c@{}c|@{}c@{}c@{}c@{}c}
$K^0_{\Z/2}$ &
$K^0$ &
$\scriptstyle{(1, 1, 1)}$ & $\scriptstyle{(-1, 1, 1)}$ &
$\scriptstyle{(1, -1, 1)}$ & $\scriptstyle{(-1, -1, 1)}$ &
$\scriptstyle{(1, 1, -1)}$ & $\scriptstyle{(-1, 1, -1)}$ &
$\scriptstyle{(1,-1,-1)}$ & $\scriptstyle{(-1, -1, -1)}$ \\
\hline
$\uC_0$ & $\uC$ &
1 & 1 & 1 & 1 &
1 & 1 & 1 & 1 \\
$\uC_1$ & $\uC$ &
$t$ & $t$ &$ t$ & $t$ &
$t$ & $t$ & $t$ & $t$ \\
\hline
$H_{12}$ & $H_{12}$ &
$t$ & 1 & 1 & 1 &
$t$ & 1 & 1 & 1 \\
$H_{23}$ & $H_{23}$ &
$t$ & $t$ & 1 & 1 &
1 & 1 & 1 & 1 \\
$H_{13}$ & $H_{13}$ &
$t$ & 1 & $t$ & 1 &
1 & 1 & 1 & 1 \\
\hline
$\uC_1H_{12}$ & $H_{12}$ &
1 & $t$ & $t$ & $t$ &
1 & $t$ & $t$ & $t$ \\
$\uC_1H_{23}$ & $H_{23}$ &
1 & 1 & $t$ & $t$ &
$t$ & $t$ & $t$ & $t$ \\
$\uC_1H_{13}$ & $H_{13}$ &
1 & $t$ & 1 & $t$ &
$t$ & $t$ & $t$ & $t$ \\
\hline
$L_1$ & $\uC$ &
$t$ & 1 & $t$ & 1 &
$t$ & 1 & $t$ & 1 \\
$L_2$ & $\uC$ &
$t$ & $t$ & 1 & 1 &
$t$ & $t$ & 1 & 1 \\
$L_3$ & $\uC$ &
$t$ & $t$ & $t$ & $t$ &
1 & 1 & 1 & 1 \\
\hline
$H_{12}L_3$ & $H_{12}$ &
1 & $t$ & $t$ & $t$ &
$t$ & 1 & 1 & 1\\
\end{tabular}
\end{table}

\subsubsection{$K^0_{\Z/2}(\tS^2 \times \tT^2)$}
\label{Sect3.3.5}
To compute the group $K^0_{\Z/2}(\tS^2 \times \tT^2)$, we start from the group $K^n_{\Z/2}(\tS^2)$.

\begin{lemma}
\label{lem3.7}
The group $K^0_{\Z/2}(\tS^2)$ is isomorphic to $\Z^4$, whose generators are given by $\{ \uC_0, \uC_1, \tH, \uC_1 \tH \}$.
We also have that $K^1_{\Z/2}(\tS^2) = 0$.
\end{lemma}

\begin{proof}
Note that $\tS^2 = \tT \times \tT / \tT \vee \tT$, as in (\ref{q}).
Based on Sects.~\ref{Sect3.3.2} and \ref{Sect3.3.3}, the result follows from the six-term exact sequence for the pair $(\tT \times \tT, \tT \vee \tT)$.
\end{proof}

Let $\pi_{12}$ and $\pi_3$ be projections from $\tS^2 \times \tT$ to $\tS^2$ and $\tT$, respectively.
We set $\tH_{12} = \pi_{12}^* \tH$ and $L_3 = \pi_3^* L$.

\begin{lemma}
\label{lem3.8}
The group $K^0_{\Z/2}(\tS^2 \times \tT)$ is isomorphic to $\Z^6$, whose generators are given by $\{ \uC_0, \uC_1, \tH_{12}, \uC_1 \tH_{12}, L_3, \tH_{12} L_3 \}$.
We also have that $K^1_{\Z/2}(\tS^2 \times \tT) =0$.
\end{lemma}

\begin{proof}
By using $U$ and $V$ in (\ref{UV}), we set $U' = \tS^2 \times U$ and $V' = \tS^2 \times V$.
Let us consider the Mayer--Vietoris exact sequence for $\{ U', V' \}$.
By Lemma~\ref{lem3.7}, we obtain the following exact sequence:
\begin{equation*}
	0 \to K^0_{\Z/2}(\tS^2 \times \tT) \overset{\varphi_{\Z/2}}{\longrightarrow} K^0_{\Z/2}(U') \oplus K^0_{\Z/2}(V') \overset{\Delta_{\Z/2}}{\longrightarrow} K^0_{\Z/2}(U' \cap V') \to K^1_{\Z/2}(\tS^2 \times \tT) \to 0.
\end{equation*}
Note that $K^0_{\Z/2}(U')$ and $K^0_{\Z/2}(V')$ are isomorphic to $K^0_{\Z/2}(\tS^2)$.
The group $K^0_{\Z/2}(U' \cap V')$ is isomorphic to $K^0(\SSS^2) \cong \Z^2$, which is generated by $\uC$ and $\tH$.
The map $\Delta_{\Z/2}$ is surjective, and its kernel is isomorphic to $\Z^6$, which computes $K^n_{\Z/2}(\tS^2 \times \tT)$.
The kernel of $\Delta_{\Z/2}$ is generated by the following elements:
\begin{equation*}
	(\uC_0, \uC_0), \ (\uC_1, \uC_1), \ (\tH, \tH), \ (\uC_1 \tH, \uC_1 \tH), \ (\uC_0, \uC_1), \ (\tH, \uC_1 \tH),
\end{equation*}
The elements of the group $K^0_{\Z/2}(\tS^2 \times \tT)$ in the statement map to the above pairs via the map $\varphi_{\Z/2}$ and therefore provide the generators of the group $K^0_{\Z/2}(\tS^2 \times \tT)$.
\end{proof}

Let $\pi_{12}$, $\pi_3$ and $\pi_4$ be projections from $\tS^2 \times \tT \times \tT$ to $\tS^2$ and the second and third factors of $\tT$, respectively.
We set $\tH_{12} = \pi_{12}^* \tH$, $H_{34} = (\pi_3, \pi_4)^*H$ and $L_i = \pi_i^* L$ for $i=3,4$.

\begin{lemma}
\label{S2S1S1}
The group $K^0_{\Z/2}(\tS^2 \times \tT \times \tT)$ is isomorphic to $\Z^{12}$, whose generators are given by
\begin{equation}
\label{generators2}
	\{ \uC_0, \uC_1, \tH_{12}, \uC_1\tH_{12}, H_{34}, \uC_1H_{34}, L_3, L_4,
	\tH_{12} L_3, \tH_{12} L_4, \tH_{12} H_{34}, \uC_1 \tH_{12} H_{34}\}. \hspace{-2mm}
\end{equation}
We also have that
$K^1_{\Z/2}(\tS^2 \times \tT \times \tT) = 0$.
\end{lemma}

\begin{proof}
The proof proceeds as in \cite[Lemma~4.4]{Gomi15}.
Let $K$ be the free abelian group generated by the twelve elements in (\ref{generators2}), and let $\iota \colon K \to K^0_{\Z/2}(\tS^2 \times \tT \times \tT)$ be the natural homomorphism.
By using $U$ and $V$ in (\ref{UV}), let $U'' = \tS^2 \times \tT \times U$ and $V'' = \tS^2 \times \tT \times V$.
By Lemma~\ref{lem3.8}, the Mayer--Vietoris sequence for the pair $\{ U'', V'' \}$ leads to the following exact sequence:
\begin{align*}
0 \to K^1_{\Z/2}(U'' \cap V'') \overset{\xi_{\Z/2}}{\longrightarrow} &K^0_{\Z/2}(\tS^2 \times \tT \times \tT) \overset{\varphi_{\Z/2}}{\longrightarrow} K^0_{\Z/2}(U'') \oplus K^0_{\Z/2}(V'')\\
	&\overset{\Delta_{\Z/2}}{\longrightarrow} K^0_{\Z/2}(U'' \cap V'') \to K^1_{\Z/2}(\tS^2 \times \tT \times \tT) \to 0.
\end{align*}
Note that $K^0_{\Z/2}(U'')$ and $K^0_{\Z/2}(V'')$ are isomorphic to $K^0_{\Z/2}(\tS^2 \times \tT) \cong \Z^6$.
The group $K^0_{\Z/2}(U'' \cap V'')$ is isomorphic to $K^0(\SSS^2 \times \T) \cong \Z^2$, which is generated by $\uC$ and $\tH_{12}$.
The map $\Delta_{\Z/2}$ is surjective; therefore, $K^1_{\Z/2}(\tS^2 \times \tT \times \tT) = 0$.
The kernel of $\Delta_{\Z/2}$ is isomorphic to $\Z^{10}$.
By using the results contained in Sect.~\ref{Sect3.3.3}, we can compute the map $\varphi_{\Z/2} \circ \iota$.
For example, we have
\begin{gather*}
	\varphi_{\Z/2} \circ \iota(\tH_{12}) = (\tH_{12}, \tH_{12}), \ \ \
	\varphi_{\Z/2} \circ \iota(H_{34}) = (\uC_0, L_3),\\
	\varphi_{\Z/2} \circ \iota(L_4) = (\uC_0, \uC_1),\ \ \
	\varphi_{\Z/2} \circ \iota(\tH_{12}H_{34}) = (\tH_{12}, \tH_{12}L_3).
\end{gather*}
As a result, we see that the map $\varphi_{\Z/2} \circ \iota$ is a surjection from $K$ to $\Ker \Delta_{\Z/2}$, and its kernel $K'$ is isomorphic to $\Z^2$ generated by the following elements:
\begin{equation*}
	-\uC_0 - \uC_1 + H_{34} + \uC_1 H_{34}, \ \ -\tH_{12} - \uC_1 \tH_{12} + \tH_{12} H_{34} + \uC_1 \tH_{12} H_{34}.
\end{equation*}
We consider the following commutative diagram of exact rows:
\[\xymatrix{
0 \ar[r]& K' \ar@{^{(}->}[r] \ar[d]_{\iota'} & K \ar[r]^{\varphi_{\Z/2} \circ \iota \hspace{3mm}} \ar[d]^{\iota} & \Ker \Delta_{\Z/2} \ar[r] \ar@{=}[d] & 0\\
0 \ar[r]& K^1_{\Z/2}(U'' \cap V'') \ar[r]^{\xi_{\Z/2}} \ar[d]_{f_{U'' \cap V''}} & K^0_{\Z/2}(\tS^2 \times \tT^2) \ar[r]^{\hspace{2mm}\varphi_{\Z/2}} \ar[d]^{f_{\SSS^2 \times \T^2}} & \Ker \Delta_{\Z/2} \ar[r] \ar[d]^f & 0\\
	& K^1(U'' \cap V'') \ar[r]^{\xi} & K^0(\SSS^2 \times \T^2) \ar[r]^{\hspace{5mm}\varphi} & \Ker \Delta \ar[r] & 0
}\]
where the map $\iota'$ is induced from $\iota$, the third row comes from the Mayer--Vietoris sequence for (nonequivariant) $K$-theory, and $f_{U'' \cap V''}$, $f_{\SSS^2 \times \T^2}$ and $f$ are forgetful maps.
By computing images of the generators of $K'$, we find that
\begin{align*}
	\xi \circ f_{U'' \cap V''}(\iota'(K')) &= f_{\SSS^2 \times \T^2} \circ \iota (K')\\
		&= 2\Z (\uC - H_{34}) \oplus 2\Z \tH_{12}(\uC - H_{34})\\
		&= \xi \circ f_{U'' \cap V''} (K^1_{\Z/2}(U'' \cap V'')).
\end{align*}
Since $\xi \circ f_{U'' \cap V''}$ is injective, we have that $\iota'(K') = K^1_{\Z/2}(U'' \cap V'')$.
Since $K^1_{\Z/2}(U'' \cap V'')$ is isomorphic to $\Z^2$, the map $\iota'$ is bijective.
By the five lemma, $\iota$ is bijective.
\end{proof}

The group $K^n_{\Z/2}(\tT \times \tS^2 \times \tT)$ is computed similarly.
The result is as follows.
\begin{lemma}
\label{S1S2S1}
The group $K^0_{\Z/2}(\tT \times \tS^2 \times \tT)$ is isomorphic to $\Z^{12}$, whose generators are given by
\begin{equation*}
	\{ \uC_0, \uC_1, \tH_{23}, \uC_1\tH_{23}, H_{14}, \uC_1H_{14}, L_1, L_4, \tH_{23} L_1, \tH_{23} L_4, \tH_{23} H_{14}, \uC_1 \tH_{23} H_{14} \}.
\end{equation*}
We also have that $K^1_{\Z/2}(\tT \times \tS^2 \times \tT) = 0$.
\end{lemma}

\subsubsection{The map $\psi_1$}
\label{Sect3.3.6}
We next compute the map
$\psi_1 = (k \times 1)^* \colon K^0_{\Z/2}(\tS^2 \times \tT^2) \to K^0_{\Z/2}(\tT^3)$
starting from the following lemma.
Let $k \colon \tT \to \tS^2$ be the inclusion.
\begin{lemma}
\label{lemma1}
The map $k^* \colon K^0_{\Z/2}(\tS^2) \to K^0_{\Z/2}(\tT)$ satisfies $k^*(\tH) = L$.
\end{lemma}
\begin{proof}
At the $\Z/2$-fixed points on $\tS^2$, the fibers of the bundle $\tH$ are $\tH_1 \cong \C_1$ and $\tH_{-1} \cong \C_0$.
Therefore, we have that $F_1 (k^*(\tH)) = (\underline{\C}; t, 1)$, which is the same as $F_1(L)$ (see the table in \cite[Sect.~4.2]{Gomi15}).
Since $F_1$ is injective, $k^*(\tH) = L$ holds.
\end{proof}

In what follows, we write $\thh_{ij} = \uC_0 - \tH_{ij}$, $h_{ij} = \uC_0 - H_{ij}$ and $l_j = \uC_0 - L_j$ to simplify the notations.

\begin{lemma}
\label{psi1}
The map $\psi_1$ satisfies
\begin{align*}
\psi_1(1) &= 1, &
\psi_1(t) &= t, &
\psi_1(\thh_{12}) &= l_1,
\\
\psi_1(t \thh_{12}) &= -l_1, &
\psi_1(h_{34}) &= h_{23}, &
\psi_1(t h_{34}) &= t h_{23},
\\
\psi_1(l_3) &= l_2, &
\psi_1(l_4) &= l_3, &
\psi_1(\thh_{12}l_3) &= (1-t)h_{12},
\\
\psi_1(\thh_{12}l_4) &= (1-t)h_{13}, &
\psi_1(\thh_{12} h_{34}) &= h_{12}l_3, &
\psi_1(t \thh_{12} h_{34}) &= -h_{12} l_3.
\end{align*}
\end{lemma}

\begin{proof}
Let us consider the following commutative diagram of ring homomorphisms:
\[\xymatrix{
K^0_{\Z/2}(\tS^2 \times \T^2) \ar[r]^{\hspace{3mm} \psi_1} & K^0_{\Z/2}(\tT^3) \\
K^0_{\Z/2}(\tS^2) \ar[u]^{\pi_{12}^*} \ar[r]^{k^*} & K^0_{\Z/2}(\tT) \ar[u]_{\pi_1^*}.
}\]
According to Lemma~\ref{lemma1}, $\psi_1(\thh_{12}) = l_1$ holds.
By considering the ring homomorphism $(\pi_i,\pi_j)^* \colon K^0_{\Z/2}(\tT^2) \to K^0_{\Z/2}(\tT^3)$ for $1 \leq i < j \leq 3$ and (\ref{eq2}), we obtain the relation $l_i l_j = (1-t)h_{ij}$ in $K^0_{\Z/2}(\tT^3)$.
By using (\ref{eq1}) and (\ref{eq3}), we have
\begin{align*}
	\psi_1(t \thh_{12}) &= tl_1 = -l_1,\\
	\psi_1(\thh_{12}l_3) &= \psi_1(\thh_{12})\psi_1(l_3) = l_1 l_2 = (1-t)h_{12},\\
	\psi_1(\thh_{12}h_{34}) &= \psi_1(\thh_{12}) \psi_1(h_{34}) = l_1 h_{23} = h_{12}l_3.
\end{align*}
The other results follow in a similar manner.
\end{proof}

The map $\psi_2 := (1 \times k \times 1)^* \colon K^0_{\Z/2}(\tT \times \tS^2 \times \tT) \to K^0_{\Z/2}(\tT^3)$ is computed in a similar way, and the result is as follows.

\begin{lemma}
\label{psi2}
The map $\psi_2$ satisfies
\begin{align*}
\psi_2(1) &= 1, &
\psi_2(t) &= t, &
\psi_2(\thh_{23}) &= l_2,
\\
\psi_2(t \thh_{23}) &= -l_2, &
\psi_2(h_{14}) &= h_{13}, &
\psi_2(t h_{14}) &= t h_{13},
\\
\psi_2(l_1) &= l_1, &
\psi_2(l_4) &= l_3, &
\psi_2(\thh_{23}l_1) &= (1-t)h_{12}, \\
\psi_2(\thh_{23}l_4) &= (1-t)h_{23}, &
\psi_2(\thh_{23} h_{14}) &= h_{12}l_3, &
\psi_2(t \thh_{23} h_{14}) &= -h_{12}l_3.
\end{align*}
\end{lemma}

\subsubsection{The map $\phi_1$}
\label{Sect3.3.7}
For the map $\phi_1$ in Fig.~\ref{diagram3D}, the equivariant $K$-groups of the $\Z/2$-space $X \times \tT$ are as follows:

\begin{lemma}
\label{equivXT}
$K^0_{\Z/2}(X \times \tT) \cong \Z^{13}, \ \ K^1_{\Z/2}(X \times \tT) \cong \Z.$
\end{lemma}

\begin{proof}
Note that
$X \times \tT = (\tS^2 \times \tT \times \tT) \cup (\tT \times \tS^2 \times \tT)$
where $\tS^2 \times \tT \times \tT$ and $\tT \times \tS^2 \times \tT$ are $\Z/2$-subspaces of $X \times \tT$.
We apply the Mayer--Vietoris exact sequence to compute the groups.
\[\xymatrix{
K^0_{\Z/2}(X \times \tT) \ar[r]^{\phi_1 \oplus \phi_2 \hspace{2cm}}& K^0_{\Z/2}(\tS^2 \times \tT \times \tT) \oplus K^0_{\Z/2}(\tT \times \tS^2 \times \tT) \ar[r]^{\hspace{2.2cm}\psi_1 - \psi_2} & K^0_{\Z/2}(\tT^3) \ar[d]
\\
K^1_{\Z/2}(\tT^3) \ar[u] &  K^1_{\Z/2}(\tS^2 \times \tT \times \tT) \oplus K^1_{\Z/2}(\tT \times \tS^2 \times \tT) \ar[l] & K^1_{\Z/2}(X \times \tT) \ar[l]
}\]
Note that $K^1_{\Z/2}(\tT^3) = 0$ \cite[Proposition~4.1]{Gomi15}.
By Lemmas~\ref{S2S1S1}, \ref{S1S2S1}, \ref{psi1} and \ref{psi2}, we obtain the result.
\end{proof}

In the proof of Lemma~\ref{equivXT}, the map $\phi_1 \oplus \phi_2$ is injective.
When we identify the group $K^0_{\Z/2}(X \times \tT)$ as the subgroup of $K^0_{\Z/2}(\tS^2 \times \tT \times \tT) \oplus K^0_{\Z/2}(\tT \times \tS^2 \times \tT)$ through the map $\phi_1 \oplus \phi_2$, a set of generators for the group $K^0_{\Z/2}(X \times \tT)$ is given by:
\begin{align*}
x_1 &= (1,1), &
x_2 &= (t, t), &
x_3 &= (\thh_{12}, l_1),
\\
x_4 &= ((1+t)\thh_{12},0), &
x_5 &= (l_3, \thh_{23}), &
x_6 &= (0, (1+t)\thh_{23}),
\\
x_7 &= (\thh_{12}l_3, l_1 \thh_{23}), &
x_8 &= (l_4, l_4), &
x_9 &= (\thh_{12} l_4, (1-t)h_{14}),
\\
x_{10} &= ((1-t)h_{34}, \thh_{23}l_4), &
x_{11} &= (\thh_{12}h_{34}, -t \thh_{23}h_{14}), & &
\\
x_{12} &= ((1+t)\thh_{12}h_{34},0),
& x_{13} &= (0, (1+t) \thh_{23}h_{14}). &
\end{align*}

\begin{remark}
\label{rem3.14}
Note that the group $K^n_{\Z/2}(X)$ is also computed by the Mayer--Vietoris sequence, and the results are $K^0_{\Z/2}(X) \cong \Z^7$ and $K^1_{\Z/2}(X) \cong \Z$.
Through the decomposition in (\ref{isom}),
\begin{equation*}
	K^0_{\Z/2}(X \times \tT) \cong K^0_{\Z/2}(X \times \{ -1 \}) \oplus K^0_{\Z/2}(X \times \tT, X \times \{-1\})
		\cong K^0_{\Z/2}(X) \oplus K^1_\pm(X).
\end{equation*}
The elements $x_1, \ldots, x_7$ correspond to the generators of $K^0_{\Z/2}(X)$, and $x_8, \ldots, x_{13}$ correspond to those of $K^1_\pm(X)$. We will use them in Sect.~\ref{Sect4}.
\end{remark}

\subsubsection{Top horizontal maps}
Combined with the above results, we compute the composite of the top horizontal maps in Fig.~\ref{diagram3D}.

\begin{lemma}
\label{image3D}
The images of each of the generators $x_i$ of $K^0_{\Z/2}(X \times \tT)$ under the map $\psi_1 \circ \phi_1$ are as follows.
\begin{align*}
\psi_1 \circ \phi_1 (x_1) &= 1, &
\psi_1 \circ \phi_1 (x_2) &= t, &
\psi_1 \circ \phi_1 (x_3) &= l_1,
\\
\psi_1 \circ \phi_1 (x_5) &= l_2, &
\psi_1 \circ \phi_1 (x_7) &= (1-t)h_{12}, &
\psi_1 \circ \phi_1 (x_8) &= l_3,
\\
\psi_1 \circ \phi_1 (x_9) &= (1-t)h_{13}, &
\psi_1 \circ \phi_1 (x_{10}) &= (1-t)h_{23}, &
\psi_1 \circ \phi_1 (x_{11}) &= h_{12} l_3,
\end{align*}
and $\psi_1 \circ \phi_1 (x_i) = 0$ for $i=4,6,12,13$.
\end{lemma}

\begin{proof}
Since we take the generators $x_i$ of $K^0_{\Z/2}(X \times \tT)$ through the injection $\phi_1 \oplus \phi_2$, the map $\phi_1$ corresponds to the first projection for $x_i$.
The result then follows from Lemma~\ref{psi1}.
\end{proof}

\begin{lemma}
\label{02}
For the generators $x_i$ of the group $K^0_{\Z/2}(X \times \tT)$, we have,
\begin{equation*}
	\mu_{3D} \circ \psi_1 \circ \phi_1 (x_{11}) = 2,
\end{equation*}
and $\mu_{3D} \circ \psi_1 \circ \phi_1 (x_i) = 0$ for $i \neq 11$.
\end{lemma}

\begin{proof}
The result is obtained by computing images of the elements in Lemma~\ref{image3D} under the map $F_3$.
For example, we have
\begin{align*}
	F_3((1-t)h_{12}) &= F_3(h_{12})  - F_3(th_{12})\\
		&= (h_{12} ; 1-t, 0,0,0, 1-t, 0,0,0) - (h_{12}; t-1, 0,0,0, t-1, 0,0,0)\\
		&= (0; 2(1-t), 0,0,0, 2(1-t), 0,0,0).
\end{align*}
Therefore, $\mu_{3D} \circ \psi_1 \circ \phi_1 (x_7) = \mu_{3D}((1-t)h_{12}) = 4 = 0$ in $\Z/4$.
By the computation of $F_3(h_{12} l_3)$ in Sect.~\ref{Sect3.3.4}, we have that $\mu_{3D} \circ \psi_1 \circ \phi_1 (x_{11}) = \mu_{3D}(h_{12} l_3) = 2$.
The images for the other generators are computed similarly (see Table~\ref{table3}).
\end{proof}

By Lemma~\ref{02}, the map $\frac{1}{2}\mu_{3D}$ in (\ref{1/2}) is surjective.

\subsection{Computation of other maps}
\label{Sect3.4}
We next focus on the other maps shown in Fig.~\ref{diagram3D}.
A part of this discussion is also used in Sect.~\ref{Sect4}, and the necessary results are included.
We first discuss some (nonequivariant) $K$-groups.

\begin{lemma}
\label{KX}
	$K^0(X) \cong \Z^3$, $K^1(X) \cong \Z^3$.
\end{lemma}

\begin{proof}
Since $X = \SSS^2 \times \T \cup \T \times \SSS^2$, the result is obtained from the Mayer--Vietoris exact sequence for the pair $\{ \SSS^2 \times \T, \T \times \SSS^2\}$.
\end{proof}

By Lemma~\ref{KX}, it follows that
\begin{equation}
\label{XtimesT}
	K^0(X \times \T) \cong K^0(X) \oplus K^1(X) \cong \Z^6.
\end{equation}
To compute the maps in Fig.~\ref{diagram3D}, it is convenient to fix an identification of $K^0(X \times \T)$ with $\Z^6$.
We use the above decomposition (\ref{XtimesT}) and fix the identifications of $K^0(X)$ and $K^1(X)$ with $\Z^3$.
For the group $K^0(X)$, a part of the Mayer--Vietoris exact sequence for $\{ \SSS^2 \times \T, \T \times \SSS^2 \}$ provides an injection $K^0(X) \to K^0(\SSS^2 \times \T) \oplus K^0(\T \times \SSS^2)$.
By using this injection, we identify $K^0(X)$ with its image and take the generators of $K^0(X)$ as $(1,1)$, $(\thh_{12}, 0)$ and $(0, \thh_{23})$.
Here $\thh_{12} = \uC - H_{12}$ follows the notation in Lemma~\ref{lem3.8}, and $\thh_{23} = \uC - H_{23}$, where $H_{23} = \pi_{23}^* \tH$ for the projection $\pi_{23} \colon \T \times \SSS^2 \to \SSS^2$.
In this way, we identify $K^0(X)$ with $\Z^3$.
We next discuss $K^1(X)$.
On the Riemann sphere $\SSS^2 = \C \cup \{ \infty \}$, we consider the orientation induced by its complex structure.
The disks $\D_+$ and $\D_-$ are also oriented accordingly as subspaces of $\SSS^2$.
We equip $\T$ with a counterclockwise orientation and the spaces $\D_\pm \times \T$ and $\T \times \D_\pm$ with product orientations.
We also equip the spaces $\SSS^3_{\pm, \pm}$ with the boundary orientation of $\D_\pm \times \D_\pm$, where the double signs correspond.
Note that the boundary orientation on $\partial \D_+ = \T$ is counterclockwise, but that on $\partial \D_-$ is clockwise.
Correspondingly, our orientation on $\T \times \D_+$ is opposite to the subspace orientation of $\T \times \D_+ \subset \SSS^3_{-+}$.
Note that the space $X$ is the union of the four solid tori (\ref{X}).
In the interior of $\D_+ \times \T$, we take a small open three-ball $B_1$ with a subspace orientation.
The quotient map $X \to X/ (X \setminus B_1) \cong \overline{B}_1 / \partial \overline{B}_1$ induces the homomorphism $e_1 \colon K^1(\overline{B}_1, \partial \overline{B}_1) \to K^1(X)$.
Similarly, we take small open three-balls $B_2$ and $B_3$ in the interiors of $\T \times \D_+$ and $\D_- \times \T$ with subspace orientations, respectively, and define the map $e_i \colon K^1(\overline{B}_i, \partial \overline{B}_i) \to K^1(X)$ for $i=2,3$.

\begin{lemma}
\label{mape}
The map $e = \sum_{i=1}^3 e_i \colon \oplus_{i=1}^3 K^1(\overline{B}_i, \partial \overline{B}_i) \to K^1(X)$ is an isomorphism.
\end{lemma}

\begin{proof}
Let $i_{\pm \pm} \colon \SSS_{\pm\pm}^3 \hookrightarrow X$ be the inclusion where double signs correspond (see (\ref{X}) and (\ref{Spm})), and
let us consider the following maps:
\begin{equation}
\oplus_{i=1}^3 K^1(\overline{B}_i, \partial \overline{B}_i) \overset{e}{\to}
	K^1(X) \overset{i}{\to}
		K^1(\SSS^3_{++}) \oplus K^1(\SSS^3_{-+}) \oplus K^1(\SSS^3_{--})
\end{equation}
where $i = i_{++}^* \oplus i_{+-}^* \oplus i_{--}^*$.
By using the orientations of $B_i$ and $\SSS^3_{\pm \pm}$, both groups $\oplus_{i=1}^3 K^1(\overline{B}_i, \partial \overline{B}_i)$ and $K^1(\SSS^3_{++}) \oplus K^1(\SSS^3_{-+}) \oplus K^1(\SSS^3_{--})$ are identified with $\Z^3$ via the three-dimensional winding numbers.
Through these identifications, the map $i \circ e$ is computed as follows\footnote{Note that our open three-ball $B_i$ is contained in some solid torus, and that each of our three-spheres $\SSS^3_{\pm\pm}$ is the union of two solid tori. The orientation of the three-ball $B_i$ can be the same as or opposite to that of the three-sphere $\SSS^3_{\pm\pm}$, which leads to a positive or negative sign in (\ref{sign}), respectively.}:
\begin{equation}
\label{sign}
	i \circ e (p,q,r) = (p+q, -q+r, -r),
\end{equation}
where $p,q,r \in \Z$, which is a bijection.
Since $i$ is surjective and $K^1(X)$ is isomorphic to $\Z^3$ by Lemma~\ref{KX}, the map $i$ is bijective.
Therefore, $e$ is also bijective.
\end{proof}

We identify $K^1(X)$ with $\Z^3$ through the isomorphism $e$,
and combining this with the decomposition (\ref{XtimesT}), we identify $K^0(X \times \T)$ with $\Z^6$.

Let $W_3 \colon K^1(X) \to \Z^4$ be the composite of the following maps:
\begin{equation*}
	K^1(X) \overset{i'}{\to} K^1(\SSS^3_{++}) \oplus K^1(\SSS^3_{-+}) \oplus K^1(\SSS^3_{--}) \oplus K^1(\SSS^3_{+-}) \overset{\cong}{\longrightarrow} \Z^4.
\end{equation*}
where the first map is provided by restriction homomorphisms $i' = i_{++}^* \oplus i_{-+}^* \oplus i_{--}^* \oplus i_{+-}^*$ and the second isomorphism is given by taking three-dimensional winding numbers for each of the four odd $K$-groups of the three-sphere.
Furthermore, let $\varphi'_1 \colon K^1(X) \to K^1(\SSS^2 \times \T)$ be the restriction homomorphism.
Through the isomorphism
$K^1(\SSS^2 \times \T) \cong K^1(\T) \oplus K^1_\cpt(\R^2 \times \T) \cong \Z \oplus \Z$,
the group $K^1(\SSS^2 \times \T)$ is identified with $\Z^2$.

\begin{lemma}
\label{lem3.20}
Let $p,q,r \in \Z$.
Based on the above identifications, the maps $W_3$ and $\varphi'_1$ are written as follows.
\begin{enumerate}
\item $W_3(p,q,r) = (p+q, -q+r, -r, -p)$.
\item $\varphi'_1(p,q,r) = (0,p+r)$.
\end{enumerate}
\end{lemma}

\begin{proof}
The group $K^1(X)$ is identified with $\Z^3$ via the isomorphism $e$.
The results then follow as in Lemma~\ref{mape}.
\end{proof}

Similarly, following the notations in Lemma~\ref{S2S1S1}, we take $1$, $\thh_{12}$, $h_{34}$ and $\thh_{12}h_{34}$ as the generators of the group $K^0(\SSS^2 \times \T^2) \cong \Z^4$.
Each element of this group can then be represented by a four-tuple of integers.

\begin{lemma}
\label{lem3.21}
Let $\alpha_i$, $p,q,r \in \Z$.
Based on the above identifications, the maps $c_2$, $\varphi_1$ and $f_2 \circ \phi_1$ are written as follows.
\begin{enumerate}
\item $c_2(\alpha_1, \alpha_2, \alpha_3, p,q,r) = (p+q, -q+r, -r, -p)$.
\item $\varphi_1(\alpha_1, \alpha_2, \alpha_3, p,q,r) = (\alpha_1, \alpha_2, 0,p+r)$.
\item $f_2 \circ \phi_1 \bigl(\sum_{i=1}^{13}a_i x_i \bigl) = (a_1 + a_2, a_3 + 2 a_4, 0, a_{11} + 2a_{12})$.
\end{enumerate}
\end{lemma}

\begin{proof}
(1) and (2) follow from Lemma~\ref{lem3.20}.
For (3), it is sufficient to compute the images of the generators $x_i$ of the group $K^0_{\Z/2}(X \times \tT)$ in Sect.~\ref{Sect3.3.7} under the map $f_2 \circ \phi_1$.
For example, the images of $x_{11}$ and $x_{12}$ are computed as follows:
\begin{equation*}
	f_2 \circ \phi_1(x_{11}) = f_2(\thh_{12}h_{34}) = \thh_{12}h_{34}, \ \ \
	f_2 \circ \phi_1(x_{12}) = f_2((1+t)\thh_{12} h_{34}) = 2 \thh_{12}h_{34}.
\end{equation*}
The other results follow in a similar manner.
\end{proof}

\subsection{Proof of Theorem~\ref{theorem3D}}
\label{Sect3.5}
Based on the above preliminaries, we provide a proof of Theorem~\ref{theorem3D}.
By using generators $x_i$ of the group $K^0_{\Z/2}(X \times \tT)$ in Sect.~\ref{Sect3.3.7}, we write the $K$-class $\eBloch$ in $K^0_{\Z/2}(X \times \tT)$ as $\sum_{i=1}^{13}a_i x_i$ where $a_i \in \Z$.
We further write $f_1(\eBloch)$ in $K^0(X \times \T)$ as $(\alpha_1, \alpha_2, \alpha_3, p, q, r)$ where $\alpha_i, p, q, r \in \Z$ by using the identification of $K^0(X \times \T)$ with $\Z^6$ in Sect.~\ref{Sect3.4}.
Regarding the symmetry-based indicator, by equation (\ref{mu3D}) and Lemma~\ref{02}, we have
\begin{equation*}
	\frac{1}{2}\mu_{3D}(H) = \frac{1}{2} \mu_{3D} \circ \psi_1 \circ \phi_1 (\eBloch)
		 = \sum_{i=1}^{13}\frac{a_i}{2}\mu_{3D} \circ \psi_1 \circ \phi_1(x_i) = a_{11}.
\end{equation*}
For the spectral flow of quarter-plane Toeplitz operators, by equation (\ref{sf}) and Lemma~\ref{lem3.21}, we have
\begin{equation*}
	(\sf_a(H), \sf_b(H), \sf_c(H), \sf_d(H)) = c_2 \circ f_1 (\eBloch) = (p+q, -q+r, -r, -p).
\end{equation*}
Given the commutativity of the diagram in Fig.~\ref{diagram3D} and Lemma~\ref{lem3.21}, the following are the same in the group $K^0(\SSS^2 \times \T^2)$:
\vspace{-2mm}
\begin{align*}
	f_2 \circ \phi_1 (\eBloch) &= f_2 \circ \phi_1 \Big(\sum_{i=1}^{13}a_i x_i \Big) = (a_1+a_2, a_3+2a_4, 0, a_{11}+2a_{12} ), \\
	\varphi_1 \circ f_1 (\eBloch) &= \varphi_1(\alpha_1, \alpha_2, \alpha_3, p, q, r) = (\alpha_1, \alpha_2,0, p+r).
\end{align*}
Figure~\ref{diagram3D2} below shows the correspondence of the elements.
Therefore,
\begin{equation}
\label{keyrelation}
	a_{11}+2a_{12} = p+r,
\end{equation}
holds.
By relation (\ref{keyrelation}), the following holds in $\Z/2$:
\begin{equation*}
\label{keycomp}
\begin{aligned}
	\frac{1}{2}\mu_{3D}(H) &= a_{11} = a_{11} + 2a_{12} = p+r \\
	&= (p+q) + (-q+r) = \sf_a(H) + \sf_b(H) \mod 2,
\end{aligned}
\end{equation*}
which completes the proof.
\qed

\begin{figure}[htbp]
\centering
\begin{tabular}{c}
\xymatrix{
\eBloch = \sum\limits_{i=1}^{13}a_i x_i \ar@{{|}->}[r]^{\phi_1} \ar@{{|}->}[d]^{f_1}
	& \sum\limits_{i=1}^{13}a_i \phi_1(x_i) \ar@{{|}->}[r]^{\hspace{3mm} \frac{1}{2}\mu_{3D} \circ \psi_1} \ar@{{|}->}[d]^{f_2}
		& a_{11} \\
(\alpha_1, \alpha_2, \alpha_3, p,q,r) \ar@{{|}->}[r]^{\varphi_1 \hspace{10mm}} \ar@{{|}->}[d]^{c_2}
	& \txt{$(a_1+a_2, a_3+2a_4, 0, a_{11}+2a_{12})$\\ $=(\alpha_1, \alpha_2,0, p+r) \hspace{2cm}$} & \\
(p+q, -q+r, -r, -p) & &
}
\end{tabular}
\caption{The correspondence of the elements contained in Fig.~\ref{diagram3D}}
\label{diagram3D2}
\end{figure}

\section{Bulk-corner correspondence}
\label{Sect4}
In this section, we consider two-dimensional systems and discuss the bulk-corner correspondence for inversion-symmetric second-order topological insulators that preserve chiral symmetry.

Let $N$ be a positive integer, and
let $H \colon \T^2 \to M(N, \C)$ be a two-variable rational matrix function with poles off $\T^2$.
We assume that $H$ takes values in Hermitian matrices.
We also write $H$ for the multiplication operator on the Hilbert space $l^2(\Z^2, \C^N)$ generated by $H$.
As in Sect.~\ref{Sect3}, we call $H$ the bulk Hamiltonian.
Let $\Pi$ and $I$ be $N$-by-$N$ Hermitian unitary matrices that anticommute, that is, $\Pi I = - I \Pi$.
We further assume that our Hamiltonian preserves both the inversion and chiral symmetries.
The symmetries of matrix-valued functions are as follows:
\begin{equation*}
	I H(z, w) I^* = H(\bar{z}, \bar{w}), \ \ \ \Pi H(z, w) \Pi^* = - H(z, w),
\end{equation*}
for any $(z,w)$ in $\T^2$.
We first assume that the matrix-valued function $H$ takes values in invertible matrices; that is, we assume that the bulk Hamiltonian is gapped at zero energy.
Let $\Gamma^2_i$ $(i=1,2,3,4)$ be the four fixed points of the involution on $\tT^2$.
As in Sect.~\ref{Sect3}, we introduce the bulk topological invariant and formulate the bulk-corner correspondence as follows \cite{TTM20a,OTY21,SO21}.
\begin{definition}[Symmetry-based indicator]
\label{mu2D}
$\mu_{2D}(H) = - \sum_{i=1}^4 n_-(\Gamma^2_i) \ \mod 4$.
\end{definition}
As in Remark~\ref{remark3.1}, the nonnegative integer $n_-(\Gamma^2_i)$ counts the number of occupied states with odd parities at the time-reversal invariant momentum $\Gamma^2_i$ with respect to the inversion symmetry operator $I$.
Let $H_0$ and $H_1$ be two such bulk-gapped Hamiltonians, and let $\eta_j = H_j |H_j|^{-1}$ be gradations on $\uC^N = \T^2 \times \C^N$.
Let us take $\gamma = \Pi$ and $\rho = \sqrt{-1}I \Pi$; then, $\eta_j$, $\rho$ and $\gamma$ anticommutes for each $j=0,1$.
Therefore, the quadruple $(\uC^N, \eta_j, \rho, \gamma)$ is a $c$-twisted vector bundle with $Cl_{0,1}$-action where $c = \id_{\Z/2}$ in Sect.~\ref{Sect2.3}, and the triple $(\uC^N, \eta_0, \eta_1)$ gives an element of the group $K^1_\pm(\tT^2)$.
The symmetry-based indicator provides the group homomorphism\footnote{Utilizing the notation of Remark~\ref{action}, we set $\chi = I$ and $\gamma = \Pi$. We then adjust the $\Z/2$-action $\chi$ to $\rho$ to define an element in our $K^1_\pm$-group.
In our restriction homomorphism (\ref{restriction}), the $\Z/2$-action on a vector bundle is reverted from $\rho$ to $\chi$ to define an element in the representation ring $R$.}
\begin{equation*}
	\mu_{2D} \colon K^1_{\pm}(\tT^2) \to \Z/4
\end{equation*}
given by $\mu_{2D} ([\uC^N, \eta_0, \eta_1]) = \mu_{2D} (H_0) - \mu_{2D} (H_1)$.
To study $\mu_{2D}$, we examine its relation with $\mu_{3D}$ in Sect.~\ref{Sect3} by using the map $t$ in Proposition~\ref{mapt}.
Since $K^0_{\Z/2}(\tT^2 \times J, \tT^2 \times \partial J)$ is isomorphic to $K^0_{\Z/2}(\tT^3, \tT^2 \times \{-1\})$, the inclusion $(\tT^3, \varnothing) \to (\tT^3, \tT^2 \times \{-1\})$ induces an injection
$j^* \colon  K^0_{\Z/2}(\tT^2 \times J, \tT^2 \times \partial J) \to K^0_{\Z/2}(\tT^3)$.
Composing it with the map $t_{\tT^2}$ in (\ref{t}), we obtain the injection
\begin{equation*}
	j^* \circ t_{\tT^2} \colon K^1_{\pm}(\tT^2) \to K^0_{\Z/2}(\tT^3).
\end{equation*}
Note that the group $K^1_{\pm}(\tT^2)$ is isomorphic to $\Z^6$, and that the image of the injection $j^* \circ t_{\tT^2}$ is generated by $h_{13}, th_{13}, h_{23}, th_{23}, l_3$ and $h_{12} l_3$ \cite[Proposition~4.12]{Gomi15}.
By the definition of our restriction homomorphism (\ref{restriction}) and the symmetry-based indicators, the following diagram commutes (see identification (\ref{eq4}) and Remarks~\ref{rem2.5} and ~\ref{action}):
\[\xymatrix{
K^1_\pm(\tT^2) \ar[rr]^{j^* \circ t_{\tT^2}} \ar[rd]_{\mu_{2D}}& & K^0_{\Z/2}(\tT^3) \ar[dl]^{\mu_{3D}}
\\
& \Z/4 &
}\]
By the commutativity of this diagram and Lemma~\ref{symindeven}, the following holds.

\begin{lemma}
The map $\mu_{2D}$ takes values of $0$ and $2$.
\end{lemma}

Therefore, we obtain the group homomorphism $\frac{1}{2}\mu_{2D} \colon K^1_{\pm}(\tT^2) \to \Z/2$.
The map $\frac{1}{2}\mu_{2D}$ is surjective\footnote{This is also confirmed by an example in Sect.~\ref{Sect5.1}.}, which follows from the computations in Sect.~\ref{Sect3}.

We next introduce corner topological invariants.
For $i = 1,2,3,4$ and $j = a,b,c,d$, let $H^i = T^i_{H}$ be a half-plane Toeplitz operator and let $H^j = T^j_{H}$ be a quarter-plane Toeplitz operator.
These half-plane and quarter-plane Toeplitz operators are our models of the four edges and the four corners of rectangle-shaped systems, respectively.
As in Sect.~\ref{Sect3}, we assume the following spectral gap condition in this section.

\begin{assumption}[Spectral gap condition]
\label{assumption2D}
For $i = 1,2,3,4$, half-plane Toeplitz operators $H^i$ are invertible.
\end{assumption}

Under Assumption~\ref{assumption2D}, four quarter-plane Toeplitz operators $H^j$ ($j=a,b,c,d$) are self-adjoint Fredholm operators \cite{DH71}, which are odd in the sense that each operator $H^j$ anticommutes with $\Pi$.
For $j = a,b,c,d$, we write $\ind_j(H)$ for the Fredholm index of the off-diagonal part of $H^j$; that is,
$\ind_j(H) = \mathrm{Tr}(\Pi|_{\Ker(H^j)})$.
These Fredholm indices account for the number of corner states at each of the four corners \cite{Hayashi3}.
The inversion symmetry interchanges corners $a$ with $c$ and $b$ with $d$, respectively.
Since the chiral symmetry operator $\Pi$ anticommutes with the inversion symmetry operator $I$, the following relations hold:
\begin{equation*}
	\ind_a(H) = - \ind_c(H), \ \ \ind_b(H) = - \ind_d(H).
\end{equation*}

The following is the main theorem in this section.

\begin{theorem}[Bulk-corner correspondence]
\label{theorem2D}
Under Assumption~\ref{assumption2D}, the following equality holds:
\begin{equation*}
	\frac{1}{2}\mu_{2D}(H) = \ind_a(H) + \ind_b(H) \ \mod 2.
\end{equation*}
\end{theorem}

\begin{proof}
Figure~\ref{diagram2D} presents the key diagram for the proof.
\begin{figure}[htbp]
\centering
\begin{tabular}{c}
\xymatrix{
K^1_\pm(X) \ar[r]^{\phi'_1 \hspace{3mm}} \ar[d]^{f'_1} & K^1_\pm(\tS^2 \times \tT) \ar[r]^{\hspace{3mm} \psi'_1} \ar[d]^{f'_2} & K^1_\pm(\tT^2) \ar[r]^{\hspace{3mm} \frac{1}{2}\mu_{2D}} & \Z/2\\
K^1(X) \ar[r]^{\varphi_1' \hspace{3mm}} \ar[d]^{W_3} & K^1(\SSS^2 \times \T) & &\\
\Z^4 & & &}
\vspace{2mm}
\\
\xymatrix{
\sum\limits_{i=8}^{13}a_i x_i \ar@{{|}->}[r]^{\phi'_1 \hspace{2mm}} \ar@{{|}->}[d]^{f'_1}
	& \sum\limits_{i=8}^{13}a_i \phi'_1(x_i) \ar@{{|}->}[rr]^{\hspace{10mm} \frac{1}{2}\mu_{2D} \circ \psi'_1 \hspace{5mm}} \ar@{{|}->}[d]^{f'_2}
		& & a_{11} \\
(p,q,r) \ar@{{|}->}[r]^{\varphi'_1 \hspace{5mm}} \ar@{{|}->}[d]^{W_3}
	& \txt{$(0, a_{11}+2a_{12})$\\ $=(0, p+r) \hspace{3mm}$} & & \\
(p+q, -q+r, -r, -p) & & &}
\end{tabular}
\caption{The key diagram for Theorem~\ref{theorem2D} and the correspondence of the elements}
\label{diagram2D}
\end{figure}
In this diagram, $f'_1$ and $f'_2$ are forgetful maps,
$\phi'_1$, $\varphi'_1$ and $\psi'_1$ are restriction homomorphisms, and the map $W_3$ is that introduced in Sect.~\ref{Sect3.4}.
Note that $\psi'_1 = (k \times 1)^*$ as in Sect.~\ref{Sect3.3.6}, and that the map $W_3$ is computed in Lemma~\ref{lem3.20}.
By the decomposition in (\ref{isom}), for a compact Hausdorff $\Z/2$-space $Y$, we have
\begin{equation*}
	K^0_{\Z/2}(Y \times \tT) \cong K^0_{\Z/2}(Y) \oplus K^0_{\Z/2}(Y \times J, Y \times \partial J) \cong K^0_{\Z/2}(Y) \oplus K^1_\pm(Y).
\end{equation*}
The second isomorphism is provided by the isomorphism $t_Y$ in Proposition~\ref{mapt}.
Combined with its nonequivariant analog decomposition $K^0(Y \times \T) \cong K^0(Y) \oplus K^1(Y)$, it follows that the diagram in Fig.~\ref{diagram2D} is contained in the diagram of Fig.~\ref{diagram3D} as a direct summand.

As in Sect.~\ref{Sect3.2}, under Assumption~\ref{assumption2D}, the extension
$H^E \colon X \to GL(N,\C)$ of the matrix-valued function $H$ on $\tT^2$ that takes values in Hermitian matrices is canonically associated.
By Lemmas~\ref{extension2} and \ref{extension3}, this extension satisfies
\begin{equation*}
	\Pi H^E(z,w) \Pi^* = - H^E(z,w), \ \ \
	I H^E(z,w) I^* = H^E(z^{-1}, w^{-1}),
\end{equation*}
for any $(z,w)$ in $X$.
Let us consider the constant matrix-valued function on $\tT^2$ that takes values in the matrix $I$.
This matrix-valued function preserves the inversion and chiral symmetries, and we write $\epsilon$ for this matrix-valued function.
The triple $(\uC^N, H^E|H^E|^{-1}, \epsilon^E)$ is a triple in Sect.~\ref{Sect2.3} that defines an element of $K^1_\pm(X)$.
By Remark~\ref{rem3.14}, the group $K^1_\pm(X)$, considered as a subgroup of $K^0_{\Z/2}(X \times \tT)$, is generated by $x_i$ for $i=8,9,\ldots,13$.
Therefore, the element
\begin{equation}
\label{element2D}
	[\uC^N, H^E|H^E|^{-1}, \epsilon^E] \in K^1_\pm(X)
\end{equation}
is written as $\sum_{i=8}^{13}a_i x_i$, where $a_i \in \Z$.
We also write $(p,q,r)$ for the image of this element (\ref{element2D}) by the map $f'_1$, identifying $K^1(X)$ with $\Z^3$ as in Sect.~\ref{Sect3.4}.
As in the proof of Theorem~\ref{theorem3D}, the correspondence of the element (\ref{element2D}) in our key diagram, shown in Fig.~\ref{diagram2D}, follows from the computations in Sect.~\ref{Sect3.3} and Lemma~\ref{lem3.20}.

Since the Fredholm index for the odd self-adjoint quarter-plane Toeplitz operator associated with $\epsilon$ is trivial,
the image of the above $K$-class (\ref{element2D}) under vertical maps $f'_1$ and $W_3$ corresponds to the Fredholm indices of the off-diagonal part of the quarter-plane Toeplitz operators associated with $H$ \cite{Hayashi5}; that is,
\begin{align*}
	W_3 \circ f'_1 ([\uC^N, H^E|H^E|^{-1}, \epsilon^E]) &= (\ind_a(H), \ind_b(H), \ind_c(H), \ind_d(H))\\
	&= (p+q, -q+r, -r, -p).
\end{align*}
Since the symmetry-based indicator of $\epsilon$ is also trivial, it follows that
\begin{align*}
	\frac{1}{2}\mu_{2D}(H) &= \frac{1}{2}( \mu_{2D}(H) - \mu_{2D}(\epsilon)) = a_{11} = a_{11} + 2a_{12}\\
		&= p+r = \ind_a(H) + \ind_b(H) \mod 2
\end{align*}
in $\Z/2$, which completes the proof.
\end{proof}

\begin{remark}
In \cite[Appendix B]{SO21}, the symmetry-based indicator is defined as
$z_{4,I} = \sum_{i=1}^4\frac{n_+(\Gamma_i) - n_-(\Gamma_i)}{2} \pmod4$.
Although this $z_{4,I}$ and $\mu_{2D}$ can be different for a single Hamiltonian (see Sect.~\ref{Sect5.1}, for an example), they induce the same map from the group $K^1_\pm(\T^2)$ to $\Z/4$.
Our choice $\mu_{2D}$ in Definition~\ref{mu2D} is simply due to the relation $\mu_{2D}(\epsilon) = 0$, which leads to the formula in Theorem~\ref{theorem2D}.
\end{remark}

\begin{remark}
\label{rem4.6}
While Theorems~\ref{theorem3D} and \ref{theorem2D} require the spectral gap condition for both the bulk and surface/edge models, the bulk invariants are defined solely by the bulk Hamiltonians.
In the three-dimensional case, for example, if a surface gap closes, the symmetry-based indicator remains unchanged as long as the bulk gap remains open, although the spectral flow is not defined.
When the symmetry-based indicator is two in $\Z/4$, opening the surface gap leads to the appearance of hinge states.
Note that the locations of the hinge states may change by the surface gap closing (it is stable as long as the surface gap remains open \cite{Hayashi2}).
For a more comprehensive account of the correspondence, we refer the reader to \cite{OPS24a}.
\end{remark}

\section{Examples}
\label{Sect5}
In this section, we discuss nontrivial examples of Theorems~\ref{theorem2D} and \ref{theorem3D}.
We start with a two-dimensional example in Sect.~\ref{Sect5.1} and then provide a three-dimensional example in Sect.~\ref{Sect5.2}.

\subsection{Two-dimensional example}
\label{Sect5.1}
Let $h \colon \T^2 \to GL(3,\C)$ be the following matrix-valued function.
\begin{equation*}
	h(z,w) =
{\renewcommand{\arraystretch}{1.2}
\begin{pmatrix}
zw + w & z - \frac{1}{2} & 0 \\
z - \frac{1}{2} & zw^{-1}+ \frac{1}{2}w^{-1} & 1 \\
0 & 1 & z^{-1} \\
\end{pmatrix}.}
\end{equation*}
We consider the following two-dimensional Hermitian Hamiltonian:
\begin{equation*}
	H(z,w)=
\begin{pmatrix}
0 & h(z,w)^* \\
h(z,w) & 0 \\
\end{pmatrix},
\end{equation*}
whose chiral and inversion symmetries are provided by the following matrices:
\begin{equation*}
\Pi =
\begin{pmatrix}
1_3 & 0 \\
0 & -1_3 \\
\end{pmatrix},
\ \
I =
\begin{pmatrix}
0 & 1_3 \\
1_3 & 0 \\
\end{pmatrix},
\end{equation*}
where $1_3$ denotes the three-by-three identity matrix.
The Hamiltonian $H$ satisfies the chiral and inversion symmetries.

We first show that our Hamiltonian $H$ satisfies Assumption~\ref{assumption2D}, which states that the associated half-plane Toeplitz operators $H^i$ $(i=1,2,3,4)$ are invertible.
This is equivalent to the invertibility of its off-diagonal part $T^i_{h}$ and therefore to the invertibility of the corresponding family of Toeplitz operators (e.g., $\{T^1_{h(\cdot,w)}\}_{w \in \T}$).
We confirm this by computing Wiener--Hopf factorizations of the symbols of the Toeplitz operators and verifying that their partial indices are trivial (see Sect.~\ref{Sect2.2}).
We employ an algorithm from \cite{GKS03} to obtain the factorizations and present the results.
A right Wiener--Hopf factorization concerning the $z$ variable (in the $x$-direction) is as follows:
$$h(z,w) = h^x_-(z,w) h^x_+(z,w),$$
where (we write $\alpha = \sqrt{25-16w+4w^2}$ to simplify the notations)
\begin{equation*}
	h^x_-(z,w) =
\begin{pmatrix}
1 & 0 & 0\\
0 & 1 & 0\\
\frac{-1+2w+\alpha}{6zw} & z^{-1} & 1 + \frac{-5+2w+\alpha}{4zw}
\end{pmatrix},
\end{equation*}
\begin{equation*}
	h^x_+(z,w) =
{\renewcommand{\arraystretch}{1.2}
\begin{pmatrix}
(1+z)w & -\frac{1}{2}+z & 0\\
-\frac{1}{2}+z & \frac{1+2z}{2w} & 1\\
\frac{-5-2w-\alpha}{6} & \frac{-5+4w-\alpha}{6w} & 0
\end{pmatrix}.}
\end{equation*}
This is a right canonical factorization for any $w$ in $\T$; therefore, the half-plane Toeplitz operator $T_h^1$ (and hence $H^1$) is invertible.
A right Wiener--Hopf factorization concerning the $w$ variable (in the $y$-direction) is as follows:
$$h(z,w) = h^y_-(z,w) h^y_+(z,w),$$
where
\begin{equation*}
	h^y_-(z,w) =
\begin{pmatrix}
1 & 0 & 0\\
\frac{1+2z}{(-1+2z)w} & 1 & 0\\
0 & 0 & 1
\end{pmatrix},
\ \
	h^y_+(z,w) =
\begin{pmatrix}
(1+z)w & -\frac{1}{2}+z & 0\\
\frac{1+10z}{2-4z} & 0 & 1\\
0 & 1 & z^{-1}
\end{pmatrix}.
\end{equation*}
This is a right canonical factorization for any $z$ in $\T$; therefore, the half-plane Toeplitz operator $T_h^2$ (and hence $H^2$) is invertible.
By inversion symmetry, $H^3$ and $H^4$ are also invertible (see Remark~\ref{reminv}), and Assumption~\ref{assumption2D} is satisfied.

We next compute the symmetry-based indicator $\mu_{2D}(H)$.
To simplify computations, we deform the Hamiltonian preserving the bulk gap and symmetry.
For $0 \leq r \leq 1$, let
\begin{equation*}
	h_r(z,w) =
{\renewcommand{\arraystretch}{1.2}
\begin{pmatrix}
zw + w & z - \frac{1+r}{2} & 0 \\
z - \frac{1+r}{2} & zw^{-1} + \frac{1+r}{2}w^{-1} + r & 1 \\
0 & 1 & z^{-1} \\
\end{pmatrix}.}
\end{equation*}
Each $h_r$ is a nonsingular matrix-valued function on $\T^2$.
We set
\begin{equation*}
	H_r(z,w)=
\begin{pmatrix}
0 & h_r(z,w)^* \\
h_r(z,w) & 0 \\
\end{pmatrix}.
\end{equation*}
Then, $H_0 = H$, and each $H_r$ satisfies the inversion and chiral symmetries given by $I$ and $\Pi$, respectively.
Since $H_r$ provides a homotopy between $H$ and $H_1$ preserving the bulk gap, the symmetry-based indicator for $H_1$ is the same as that for $H$.

\begin{itemize}
\item The negative eigenvalues of $H_1(1,1)$ are $-2 \pm \sqrt2, -2$, and its eigenvectors are
$^t\!(0, -1 \pm \sqrt{2}, -1, 0, 1 \mp \sqrt{2}, 1)$, $^t\!(-1,0,0,1,0,0)$ where the double sign corresponds.
Their parities are odd; therefore, $n_-(1,1)=3$.
\item The negative eigenvalues of $H_1(-1,1)$ are $\pm1-\sqrt{3},-2$, and its eigenvectors are $^t\!(\mp 1+\sqrt{3},\pm 2 - \sqrt{3},\pm1, -1 \pm \sqrt{3},2\mp \sqrt{3},1)$ and $^t\!(1,1,-1,1,1,-1)$, where the double sign correspons.
Therefore, $n_-(-1,1)=1$.
\item The negative eigenvalues of $H_1(1,-1)$ are $-2, -\sqrt{2}$, where $-\sqrt{2}$ is doubly degenerated.
Their eigenvectors are $^t\!(1,0,0,1,0,0)$, $^t\!(0, -1-\sqrt{2},1,0,-1-\sqrt{2},1)$, and $^t\!(0,-1+\sqrt{2},1,0,1-\sqrt{2},-1)$.
Therefore, $n_-(1,-1)=1$.
\item Since $H_1(-1,-1) = H_1(-1,1)$, we also have $n_-(-1,-1)=1$.
\end{itemize}
Therefore, $\frac{1}{2}\mu_{2D}(H) = \frac{1}{2}\mu_{2D}(H_1) = -3 \equiv 1 \pmod2$.
It follows from Theorem~\ref{theorem2D} that $\ind_a(H) + \ind_b(H) \equiv 1 \pmod 2$,
and topological corner states appear for our model $H$.

\subsection{Three-dimensional example}
\label{Sect5.2}
A three-dimensional example can be obtained by using the triples $H$, $\Pi$, and $I$ in the two-dimensional case of Sect.~\ref{Sect5.1} as follows.
Let $\epsilon$ be the matrix $I$.
We consider the following continuous matrix-valued function $H' \colon \T^3 \to GL(6,\C)$:
\begin{equation*}
H'(z, w, \theta) =
\left\{
\begin{aligned}
	H(z, w) \cos \theta - \Pi \sin \theta, & \hspace{3mm} \text{if} \ \ -\tfrac{\pi}{2} \leq \theta \leq \tfrac{\pi}{2},\\
	\epsilon \cos \theta - \Pi \sin \theta, & \hspace{3mm} \text{if} \ \ \ \tfrac{\pi}{2} \leq \theta \leq \tfrac{3\pi}{2}.
\end{aligned}
\right.
\end{equation*}
The inversion symmetry for $H'$ is given by the matrix $I$.
Since $H$ satisfies Assumption~\ref{assumption2D} and $\Pi$ anticommutes with $H$ and $I$, $H'$ satisfies Assumption~\ref{assumption3D}.
The eight $\Z/2$-fixed points on $\tT^3$ are $\{ \pm1, \pm1, 0 \}$ and $\{ \pm 1, \pm1, \pi \}$, on which $H'(\pm 1, \pm1, 0) = H(\pm1, \pm1)$ and $H'(\pm 1, \pm1, \pi) = -\epsilon$.
Therefore, utilizing the computations from the two-dimensional example in Sect.~\ref{Sect5.1}, we compute $\frac{1}{2}\mu_{3D}(H') = 1$.
From Theorem~\ref{theorem3D}, we find that $\sf_a(H') + \sf_b(H') = 1 \pmod 2$.

To apply Theorem~\ref{theorem3D} to other models, we need to verify Assumption~\ref{assumption3D}, even though the form of the Wiener--Hopf factorization may be much more complicated.
The applications of the method presented in this paper require further investigation, and we simply provide computable examples.

\subsection*{Acknowledgments}

The author would like to thank Takeshi Nakanishi and Ryo Okugawa for their valuable discussions.
He also thanks Ryo Takahashi for helpful information concerning symmetry-based indicators.
This work was supported by JSPS KAKENHI (Grant No. JP23K12966) and JST PRESTO (Grant No. JPMJPR19L7).

\subsection*{Declarations}

\subsection*{Funding}
This work was supported by JSPS KAKENHI (Grant No. JP23K12966) and JST PRESTO (Grant No. JPMJPR19L7).

\subsection*{Competing interests}
The author does not have any other competing interests to declare.

\subsection*{Data availability}
This article does not have any external supporting data.



\begin{thebibliography}{64}
\ifx \bisbn   \undefined \def \bisbn  #1{ISBN #1}\fi
\ifx \binits  \undefined \def \binits#1{#1}\fi
\ifx \bauthor  \undefined \def \bauthor#1{#1}\fi
\ifx \batitle  \undefined \def \batitle#1{#1}\fi
\ifx \bjtitle  \undefined \def \bjtitle#1{#1}\fi
\ifx \bvolume  \undefined \def \bvolume#1{\textbf{#1}}\fi
\ifx \byear  \undefined \def \byear#1{#1}\fi
\ifx \bissue  \undefined \def \bissue#1{#1}\fi
\ifx \bfpage  \undefined \def \bfpage#1{#1}\fi
\ifx \blpage  \undefined \def \blpage #1{#1}\fi
\ifx \burl  \undefined \def \burl#1{\textsf{#1}}\fi
\ifx \doiurl  \undefined \def \doiurl#1{\url{https://doi.org/#1}}\fi
\ifx \betal  \undefined \def \betal{\textit{et al.}}\fi
\ifx \binstitute  \undefined \def \binstitute#1{#1}\fi
\ifx \binstitutionaled  \undefined \def \binstitutionaled#1{#1}\fi
\ifx \bctitle  \undefined \def \bctitle#1{#1}\fi
\ifx \beditor  \undefined \def \beditor#1{#1}\fi
\ifx \bpublisher  \undefined \def \bpublisher#1{#1}\fi
\ifx \bbtitle  \undefined \def \bbtitle#1{#1}\fi
\ifx \bedition  \undefined \def \bedition#1{#1}\fi
\ifx \bseriesno  \undefined \def \bseriesno#1{#1}\fi
\ifx \blocation  \undefined \def \blocation#1{#1}\fi
\ifx \bsertitle  \undefined \def \bsertitle#1{#1}\fi
\ifx \bsnm \undefined \def \bsnm#1{#1}\fi
\ifx \bsuffix \undefined \def \bsuffix#1{#1}\fi
\ifx \bparticle \undefined \def \bparticle#1{#1}\fi
\ifx \barticle \undefined \def \barticle#1{#1}\fi
\ifx \bconfdate \undefined \def \bconfdate #1{#1}\fi
\ifx \botherref \undefined \def \botherref #1{#1}\fi
\ifx \url \undefined \def \url#1{\textsf{#1}}\fi
\ifx \bchapter \undefined \def \bchapter#1{#1}\fi
\ifx \bbook \undefined \def \bbook#1{#1}\fi
\ifx \bcomment \undefined \def \bcomment#1{#1}\fi
\ifx \oauthor \undefined \def \oauthor#1{#1}\fi
\ifx \citeauthoryear \undefined \def \citeauthoryear#1{#1}\fi
\ifx \endbibitem  \undefined \def \endbibitem {}\fi
\ifx \bconflocation  \undefined \def \bconflocation#1{#1}\fi
\ifx \arxivurl  \undefined \def \arxivurl#1{\textsf{#1}}\fi
\csname PreBibitemsHook\endcsname

\bibitem{AH04}
\begin{bchapter}
\bauthor{\bsnm{Atiyah}, \binits{M.}},
\bauthor{\bsnm{Hopkins}, \binits{M.}}:
\bctitle{A variant of {$K$}-theory: {$K_\pm$}}.
In: \bbtitle{Topology, Geometry and Quantum Field Theory}.
\bsertitle{London Math. Soc. Lecture Note Ser.},
vol. \bseriesno{308},
pp. \bfpage{5}--\blpage{17}.
\bpublisher{Cambridge Univ. Press}, \blocation{Cambridge}
(\byear{2004})
\end{bchapter}
\endbibitem

\bibitem{At67}
\begin{bbook}
\bauthor{\bsnm{Atiyah}, \binits{M.F.}}:
\bbtitle{{$K$}-Theory}.
\bsertitle{Lecture Notes by D. W. Anderson},
p. \bfpage{166}.
\bpublisher{W. A. Benjamin, Inc.}, \blocation{New York-Amsterdam}
(\byear{1967})
\end{bbook}
\endbibitem

\bibitem{AS69}
\begin{barticle}
\bauthor{\bsnm{Atiyah}, \binits{M.F.}},
\bauthor{\bsnm{Singer}, \binits{I.M.}}:
\batitle{Index theory for skew-adjoint {F}redholm operators}.
\bjtitle{Inst. Hautes \'Etudes Sci. Publ. Math}.
\bvolume{37},
\bfpage{5--26}
(\byear{1969})
\end{barticle}
\endbibitem

\bibitem{BHMS05}
\begin{barticle}
\bauthor{\bsnm{Baum}, \binits{P.F.}},
\bauthor{\bsnm{Hajac}, \binits{P.M.}},
\bauthor{\bsnm{Matthes}, \binits{R.}},
\bauthor{\bsnm{Szyma\'{n}ski}, \binits{W.}}:
\batitle{The {$K$}-theory of {H}eegaard-type quantum 3-spheres}.
\bjtitle{$K$-Theory}
\bvolume{35}(\bissue{1-2}),
\bfpage{159}--\blpage{186}
(\byear{2005})
\end{barticle}
\endbibitem

\bibitem{BvES94}
\begin{barticle}
\bauthor{\bsnm{Bellissard}, \binits{J.}},
\bauthor{\bsnm{Elst}, \binits{A.}},
\bauthor{\bsnm{Schulz-Baldes}, \binits{H.}}:
\batitle{The noncommutative geometry of the quantum {H}all effect}.
\bjtitle{J. Math. Phys.}
\bvolume{35}(\bissue{10}),
\bfpage{5373}--\blpage{5451}
(\byear{1994})
\end{barticle}
\endbibitem

\bibitem{BBH17a}
\begin{barticle}
\bauthor{\bsnm{Benalcazar}, \binits{W.A.}},
\bauthor{\bsnm{Bernevig}, \binits{B.A.}},
\bauthor{\bsnm{Hughes}, \binits{T.L.}}:
\batitle{Quantized electric multipole insulators}.
\bjtitle{Science}
\bvolume{357},
\bfpage{61}--\blpage{66}
(\byear{2017})
\end{barticle}
\endbibitem

\bibitem{Bre67}
\begin{bbook}
\bauthor{\bsnm{Bredon}, \binits{G.E.}}:
\bbtitle{Equivariant Cohomology Theories}.
\bsertitle{Lecture Notes in Mathematics, No. 34},
\bpublisher{Springer}, \blocation{Berlin}
(\byear{1967})
\end{bbook}
\endbibitem

\bibitem{CG81}
\begin{bbook}
\bauthor{\bsnm{Clancey}, \binits{K.F.}},
\bauthor{\bsnm{Gohberg}, \binits{I.}}:
\bbtitle{Factorization of Matrix Functions and Singular Integral Operators}.
\bsertitle{Operator Theory: Advances and Applications},
vol. \bseriesno{3},
p. \bfpage{234}.
\bpublisher{Birkh\"{a}user Verlag}, \blocation{Basel-Boston, Mass.}
(\byear{1981})
\end{bbook}
\endbibitem

\bibitem{CDS72}
\begin{barticle}
\bauthor{\bsnm{Coburn}, \binits{L.A.}},
\bauthor{\bsnm{Douglas}, \binits{R.G.}},
\bauthor{\bsnm{Singer}, \binits{I.M.}}:
\batitle{An index theorem for {W}iener--{H}opf operators on the discrete
  quarter-plane}.
\bjtitle{J. Differ. Geom.}
\bvolume{6},
\bfpage{587}--\blpage{593}
(\byear{1972})
\end{barticle}
\endbibitem

\bibitem{DH71}
\begin{barticle}
\bauthor{\bsnm{Douglas}, \binits{R.G.}},
\bauthor{\bsnm{Howe}, \binits{R.}}:
\batitle{On the {$C\sp*$}-algebra of {T}oeplitz operators on the quarterplane}.
\bjtitle{Trans. Am. Math. Soc.}
\bvolume{158},
\bfpage{203}--\blpage{217}
(\byear{1971})
\end{barticle}
\endbibitem

\bibitem{Dud77b}
\begin{barticle}
\bauthor{\bsnm{Dudu{\v c}ava}, \binits{R.V.}}:
\batitle{Discrete convolution operators on the quarter plane, and their
  indices}.
\bjtitle{Izv. Akad. Nauk SSSR Ser. Mat.}
\bvolume{41}(\bissue{5}),
\bfpage{1125}--\blpage{1137}
(\byear{1977})
\end{barticle}
\endbibitem

\bibitem{FM13}
\begin{barticle}
\bauthor{\bsnm{Freed}, \binits{D.S.}},
\bauthor{\bsnm{Moore}, \binits{G.W.}}:
\batitle{Twisted equivariant matter}.
\bjtitle{Ann. Henri Poincar\'e}
\bvolume{14}(\bissue{8}),
\bfpage{1927}--\blpage{2023}
(\byear{2013})
\end{barticle}
\endbibitem

\bibitem{FK07}
\begin{barticle}
\bauthor{\bsnm{Fu}, \binits{L.}},
\bauthor{\bsnm{Kane}, \binits{C.L.}}:
\batitle{Topological insulators with inversion symmetry}.
\bjtitle{Phys. Rev. B}
\bvolume{76},
\bfpage{045302}
(\byear{2007})
\end{barticle}
\endbibitem

\bibitem{Geier18}
\begin{barticle}
\bauthor{\bsnm{Geier}, \binits{M.}},
\bauthor{\bsnm{Trifunovic}, \binits{L.}},
\bauthor{\bsnm{Hoskam}, \binits{M.}},
\bauthor{\bsnm{Brouwer}, \binits{P.W.}}:
\batitle{Second-order topological insulators and superconductors with an
  order-two crystalline symmetry}.
\bjtitle{Phys. Rev. B}
\bvolume{97},
\bfpage{205135}
(\byear{2018})
\end{barticle}
\endbibitem

\bibitem{GKS03}
\begin{bchapter}
\bauthor{\bsnm{Gohberg}, \binits{I.}},
\bauthor{\bsnm{Kaashoek}, \binits{M.A.}},
\bauthor{\bsnm{Spitkovsky}, \binits{I.M.}}:
\bctitle{An overview of matrix factorization theory and operator applications}.
In: \bbtitle{Factorization and Integrable Systems ({F}aro, 2000)}.
\bsertitle{Oper. Theory Adv. Appl.},
vol. \bseriesno{141},
pp. \bfpage{1}--\blpage{102}.
\bpublisher{Birkh\"{a}user}, \blocation{Basel}
(\byear{2003})
\end{bchapter}
\endbibitem

\bibitem{GK58r}
\begin{barticle}
\bauthor{\bsnm{Gohberg}, \binits{I.C.}},
\bauthor{\bsnm{Kre\u{\i}n}, \binits{M.G.}}:
\batitle{Systems of integral equations on the half-line with kernels depending
  on the difference of the arguments}.
\bjtitle{Uspehi Mat. Nauk (N.S.)}
\bvolume{13}(\bissue{2 (80)}),
\bfpage{3}--\blpage{72}
(\byear{1958})
\end{barticle}
\endbibitem

\bibitem{Gomi15}
\begin{barticle}
\bauthor{\bsnm{Gomi}, \binits{K.}}:
\batitle{A variant of {$K$}-theory and topological {T}-duality for real circle
  bundles}.
\bjtitle{Comm. Math. Phys.}
\bvolume{334}(\bissue{2}),
\bfpage{923}--\blpage{975}
(\byear{2015})
\end{barticle}
\endbibitem

\bibitem{Gomi23}
\begin{barticle}
\bauthor{\bsnm{Gomi}, \binits{K.}}:
\batitle{Freed-{M}oore {$K$}-theory}.
\bjtitle{Comm. Anal. Geom.}
\bvolume{31}(\bissue{4}),
\bfpage{979}--\blpage{1067}
(\byear{2023})
\end{barticle}
\endbibitem

\bibitem{GKT21}
\begin{barticle}
\bauthor{\bsnm{Gomi}, \binits{K.}},
\bauthor{\bsnm{Kubota}, \binits{Y.}},
\bauthor{\bsnm{Thiang}, \binits{G.C.}}:
\batitle{Twisted crystallographic {T}-duality via the {B}aum-{C}onnes
  isomorphism}.
\bjtitle{Internat. J. Math.}
\bvolume{32}(\bissue{10}),
\bfpage{2150078}
(\byear{2021})
\end{barticle}
\endbibitem

\bibitem{GT19b}
\begin{barticle}
\bauthor{\bsnm{Gomi}, \binits{K.}},
\bauthor{\bsnm{Thiang}, \binits{G.C.}}:
\batitle{Crystallographic bulk-edge correspondence: glide reflections and
  twisted mod 2 indices}.
\bjtitle{Lett. Math. Phys.}
\bvolume{109}(\bissue{4}),
\bfpage{857}--\blpage{904}
(\byear{2019})
\end{barticle}
\endbibitem

\bibitem{GT19a}
\begin{barticle}
\bauthor{\bsnm{Gomi}, \binits{K.}},
\bauthor{\bsnm{Thiang}, \binits{G.C.}}:
\batitle{Crystallographic {T}-duality}.
\bjtitle{J. Geom. Phys.}
\bvolume{139},
\bfpage{50}--\blpage{77}
(\byear{2019})
\end{barticle}
\endbibitem

\bibitem{GP13}
\begin{barticle}
\bauthor{\bsnm{Graf}, \binits{G.M.}},
\bauthor{\bsnm{Porta}, \binits{M.}}:
\batitle{Bulk-edge correspondence for two-dimensional topological insulators}.
\bjtitle{Comm. Math. Phys.}
\bvolume{324}(\bissue{3}),
\bfpage{851}--\blpage{895}
(\byear{2013})
\end{barticle}
\endbibitem

\bibitem{Hat93b}
\begin{barticle}
\bauthor{\bsnm{Hatsugai}, \binits{Y.}}:
\batitle{Chern number and edge states in the integer quantum hall effect}.
\bjtitle{Phys. Rev. Lett.}
\bvolume{71}(\bissue{22}),
\bfpage{3697}--\blpage{3700}
(\byear{1993})
\end{barticle}
\endbibitem

\bibitem{Hayashi2}
\begin{barticle}
\bauthor{\bsnm{Hayashi}, \binits{S.}}:
\batitle{Topological invariants and corner states for {H}amiltonians on a
  three-dimensional lattice}.
\bjtitle{Comm. Math. Phys.}
\bvolume{364}(\bissue{1}),
\bfpage{343}--\blpage{356}
(\byear{2018})
\end{barticle}
\endbibitem

\bibitem{Hayashi3}
\begin{barticle}
\bauthor{\bsnm{Hayashi}, \binits{S.}}:
\batitle{Toeplitz operators on concave corners and topologically protected
  corner states}.
\bjtitle{Lett. Math. Phys.}
\bvolume{109}(\bissue{10}),
\bfpage{2223}--\blpage{2254}
(\byear{2019})
\end{barticle}
\endbibitem

\bibitem{Hayashi4}
\begin{barticle}
\bauthor{\bsnm{Hayashi}, \binits{S.}}:
\batitle{Classification of topological invariants related to corner states}.
\bjtitle{Lett. Math. Phys.}
\bvolume{111}(\bissue{5}),
\bfpage{118}
(\byear{2021})
\end{barticle}
\endbibitem

\bibitem{Hayashi5}
\begin{barticle}
\bauthor{\bsnm{Hayashi}, \binits{S.}}:
\batitle{An index theorem for quarter-plane {T}oeplitz operators via extended
  symbols and gapped invariants related to corner states}.
\bjtitle{Comm. Math. Phys.}
\bvolume{400},
\bfpage{429}--\blpage{462}
(\byear{2023})
\end{barticle}
\endbibitem

\bibitem{HPB11}
\begin{barticle}
\bauthor{\bsnm{Hughes}, \binits{T.L.}},
\bauthor{\bsnm{Prodan}, \binits{E.}},
\bauthor{\bsnm{Bernevig}, \binits{B.A.}}:
\batitle{Inversion-symmetric topological insulators}.
\bjtitle{Phys. Rev. B}
\bvolume{83},
\bfpage{245132}
(\byear{2011})
\end{barticle}
\endbibitem

\bibitem{Karoubi68}
\begin{barticle}
\bauthor{\bsnm{Karoubi}, \binits{M.}}:
\batitle{Alg\`ebres de {C}lifford et {$K$}-th\'eorie}.
\bjtitle{Ann. Sci. \'Ecole Norm. Sup. (4)}
\bvolume{1},
\bfpage{161}--\blpage{270}
(\byear{1968})
\end{barticle}
\endbibitem

\bibitem{Karoubi78}
\begin{bbook}
\bauthor{\bsnm{Karoubi}, \binits{M.}}:
\bbtitle{{$K$}-theory: An Introduction},
\bsertitle{Grundlehren der Mathematischen Wissenschaften, Band 226},
\bpublisher{Springer}, \blocation{Berlin}
(\byear{1978}).
\end{bbook}
\endbibitem

\bibitem{KSS20}
\begin{barticle}
\bauthor{\bsnm{Kawabata}, \binits{K.}},
\bauthor{\bsnm{Sato}, \binits{M.}},
\bauthor{\bsnm{Shiozaki}, \binits{K.}}:
\batitle{Higher-order non-hermitian skin effect}.
\bjtitle{Phys. Rev. B}
\bvolume{102},
\bfpage{205118}
(\byear{2020})
\end{barticle}
\endbibitem

\bibitem{KRSB02}
\begin{barticle}
\bauthor{\bsnm{Kellendonk}, \binits{J.}},
\bauthor{\bsnm{Richter}, \binits{T.}},
\bauthor{\bsnm{Schulz-Baldes}, \binits{H.}}:
\batitle{Edge current channels and {C}hern numbers in the integer quantum
  {H}all effect}.
\bjtitle{Rev. Math. Phys.}
\bvolume{14}(\bissue{1}),
\bfpage{87}--\blpage{119}
(\byear{2002})
\end{barticle}
\endbibitem

\bibitem{Khalaf18b}
\begin{barticle}
\bauthor{\bsnm{Khalaf}, \binits{E.}}:
\batitle{Higher-order topological insulators and superconductors protected by
  inversion symmetry}.
\bjtitle{Phys. Rev. B}
\bvolume{97},
\bfpage{205136}
(\byear{2018})
\end{barticle}
\endbibitem

\bibitem{Khalaf18a}
\begin{barticle}
\bauthor{\bsnm{Khalaf}, \binits{E.}},
\bauthor{\bsnm{Po}, \binits{H.C.}},
\bauthor{\bsnm{Vishwanath}, \binits{A.}},
\bauthor{\bsnm{Watanabe}, \binits{H.}}:
\batitle{Symmetry indicators and anomalous surface states of topological
  crystalline insulators}.
\bjtitle{Phys. Rev. X}
\bvolume{8},
\bfpage{031070}
(\byear{2018})
\end{barticle}
\endbibitem

\bibitem{Kit09}
\begin{barticle}
\bauthor{\bsnm{Kitaev}, \binits{A.}}:
\batitle{Periodic table for topological insulators and superconductors}.
\bjtitle{AIP Conf. Proc.}
\bvolume{1134}(\bissue{1}),
\bfpage{22}--\blpage{30}
(\byear{2009})
\end{barticle}
\endbibitem

\bibitem{Ku15}
\begin{barticle}
\bauthor{\bsnm{Kubota}, \binits{Y.}}:
\batitle{Controlled topological phases and bulk-edge correspondence}.
\bjtitle{Comm. Math. Phys.}
\bvolume{349}(\bissue{2}),
\bfpage{493}--\blpage{525}
(\byear{2017})
\end{barticle}
\endbibitem

\bibitem{KP74}
\begin{barticle}
\bauthor{\bsnm{Ku{\v c}ment}, \binits{P.A.}},
\bauthor{\bsnm{Pankov}, \binits{A.A.}}:
\batitle{Classifying spaces for equivariant {$K$}-theory}.
\bjtitle{Mat. Sb. (N.S.)}
\bvolume{95(137)},
\bfpage{35}--\blpage{52}
(\byear{1974})
\end{barticle}
\endbibitem

\bibitem{MW18}
\begin{barticle}
\bauthor{\bsnm{Matsugatani}, \binits{A.}},
\bauthor{\bsnm{Watanabe}, \binits{H.}}:
\batitle{Connecting higher-order topological insulators to lower-dimensional
  topological insulators}.
\bjtitle{Phys. Rev. B}
\bvolume{98},
\bfpage{205129}
(\byear{2018})
\end{barticle}
\endbibitem

\bibitem{Mat71b}
\begin{barticle}
\bauthor{\bsnm{Matumoto}, \binits{T.}}:
\batitle{Equivariant {$K$}-theory and {F}redholm operators}.
\bjtitle{J. Fac. Sci. Univ. Tokyo Sect. I A Math.}
\bvolume{18},
\bfpage{109}--\blpage{125}
(\byear{1971})
\end{barticle}
\endbibitem

\bibitem{Mat71a}
\begin{barticle}
\bauthor{\bsnm{Matumoto}, \binits{T.}}:
\batitle{On {$G$}-{${\rm CW}$} complexes and a theorem of {J}. {H}. {C}.
  {W}hitehead}.
\bjtitle{J. Fac. Sci. Univ. Tokyo Sect. IA Math.}
\bvolume{18},
\bfpage{363}--\blpage{374}
(\byear{1971})
\end{barticle}
\endbibitem

\bibitem{OTY21}
\begin{barticle}
\bauthor{\bsnm{Okugawa}, \binits{R.}},
\bauthor{\bsnm{Takahashi}, \binits{R.}},
\bauthor{\bsnm{Yokomizo}, \binits{K.}}:
\batitle{Non-Hermitian band topology with generalized inversion symmetry}.
\bjtitle{Phys. Rev. B},
\bvolume{103},
\bfpage{205205}
(\byear{2021})
\end{barticle}
\endbibitem

\bibitem{OSS19}
\begin{barticle}
\bauthor{\bsnm{Okuma}, \binits{N.}},
\bauthor{\bsnm{Sato}, \binits{M.}},
\bauthor{\bsnm{Shiozaki}, \binits{K.}}:
\batitle{Topological classification under nonmagnetic and magnetic point group
  symmetry: application of real-space {A}tiyah--{H}irzebruch spectral sequence
  to higher-order topology}.
\bjtitle{Phys. Rev. B}
\bvolume{99},
\bfpage{085127}
(\byear{2019})
\end{barticle}
\endbibitem

\bibitem{OW18}
\begin{barticle}
\bauthor{\bsnm{Ono}, \binits{S.}},
\bauthor{\bsnm{Watanabe}, \binits{H.}}:
\batitle{Unified understanding of symmetry indicators for all internal symmetry
  classes}.
\bjtitle{Phys. Rev. B}
\bvolume{98},
\bfpage{115150}
(\byear{2018})
\end{barticle}
\endbibitem

\bibitem{Pa90}
\begin{barticle}
\bauthor{\bsnm{Park}, \binits{E.}}:
\batitle{Index theory and {T}oeplitz algebras on certain cones in {${\bf Z}^2$}}.
\bjtitle{J. Oper. Theory}
\bvolume{23}(\bissue{1}),
\bfpage{125}--\blpage{146}
(\byear{1990})
\end{barticle}
\endbibitem

\bibitem{Po17}
\begin{barticle}
\bauthor{\bsnm{Po}, \binits{H.C.}},
\bauthor{\bsnm{Vishwanath}, \binits{A.}},
\bauthor{\bsnm{Watanabe}, \binits{H.}}:
\batitle{Symmetry-based indicators of band topology in the 230 space groups}.
\bjtitle{Nat. Commun}.
\bvolume{8}(\bissue{1}),
\bfpage{50}
(\byear{2017})
\end{barticle}
\endbibitem

\bibitem{Ojito23}
\begin{barticle}
\bauthor{\bsnm{Polo~Ojito}, \binits{D.}}:
\batitle{Interface currents and corner states in magnetic quarter-plane
  systems}.
\bjtitle{Adv. Theor. Math. Phys.}
\bvolume{27}(\bissue{6}),
\bfpage{1813}--\blpage{1855}
(\byear{2023})
\end{barticle}
\endbibitem

\bibitem{OPS24a}
\begin{barticle}
\oauthor{\bsnm{Polo~Ojito}, \binits{D.}},
\oauthor{\bsnm{Prodan}, \binits{E.}},
\oauthor{\bsnm{Stoiber}, \binits{T.}}:
\batitle{${C}^*$-framework for higher-order bulk-boundary correspondences}.
\bjtitle{Comm. Math. Phys.}
\bvolume{406},
\bfpage{233}
(\byear{2025})
\end{barticle}
\endbibitem

\bibitem{OPS24b}
\begin{botherref}
\oauthor{\bsnm{Polo~Ojito}, \binits{D.}},
\oauthor{\bsnm{Prodan}, \binits{E.}},
\oauthor{\bsnm{Stoiber}, \binits{T.}}:
A space-adiabatic approach for bulk-defect correspondences in lattice models of topological insulators
(2024).
preprint, {\tt arXiv:2410.24097[math-ph]}
\end{botherref}
\endbibitem

\bibitem{Prodan21}
\begin{barticle}
\bauthor{\bsnm{Prodan}, \binits{E.}}:
\batitle{Topological lattice defects by groupoid methods and Kasparov\mbox{'s} KK-theory}.
\bjtitle{J. Phys. A: Math. Theor.}
\bvolume{54}(\bissue{42}),
\bfpage{424001}
(\byear{2021})
\end{barticle}
\endbibitem

\bibitem{PS86}
\begin{bbook}
\bauthor{\bsnm{Pressley}, \binits{A.}},
\bauthor{\bsnm{Segal}, \binits{G.}}:
\bbtitle{Loop groups}.
\bsertitle{Oxford Math. Monogr.},
\bpublisher{Clarendon Press}
(\byear{1986})
\end{bbook}
\endbibitem

\bibitem{PSB16}
\begin{bbook}
\bauthor{\bsnm{Prodan}, \binits{E.}},
\bauthor{\bsnm{Schulz-Baldes}, \binits{H.}}:
\bbtitle{Bulk and boundary invariants for complex topological insulators}.
\bsertitle{Math. Phys. Stud.},
\bpublisher{Springer}, \blocation{New York}
(\byear{2016})
\end{bbook}
\endbibitem

\bibitem{Schindler18}
\begin{barticle}
\bauthor{\bsnm{Schindler}, \binits{F.}},
\bauthor{\bsnm{Cook}, \binits{A.M.}},
\bauthor{\bsnm{Vergniory}, \binits{M.G.}},
\bauthor{\bsnm{Wang}, \binits{Z.}},
\bauthor{\bsnm{Parkin}, \binits{S.S.P.}},
\bauthor{\bsnm{Bernevig}, \binits{B.A.}},
\bauthor{\bsnm{Neupert}, \binits{T.}}:
\batitle{Higher-order topological insulators}.
\bjtitle{Sci. Adv.}
\bvolume{4}(\bissue{6}),
\bfpage{0346}
(\byear{2018})
\end{barticle}
\endbibitem

\bibitem{Se68}
\begin{barticle}
\bauthor{\bsnm{Segal}, \binits{G.}}:
\batitle{Equivariant {$K$}-theory}.
\bjtitle{Inst. Hautes \'Etudes Sci. Publ. Math}.
(\bissue{34}),
\bfpage{129}--\blpage{151}
(\byear{1968})
\end{barticle}
\endbibitem

\bibitem{SO21}
\begin{barticle}
\bauthor{\bsnm{Shiozaki}, \binits{K.}},
\bauthor{\bsnm{Ono}, \binits{S.}}:
\batitle{Symmetry indicator in non-hermitian systems}.
\bjtitle{Phys. Rev. B}
\bvolume{104},
\bfpage{035424}
(\byear{2021})
\end{barticle}
\endbibitem

\bibitem{SS14}
\begin{barticle}
\bauthor{\bsnm{Shiozaki}, \binits{K.}},
\bauthor{\bsnm{Sato}, \binits{M.}}:
\batitle{Topology of crystalline insulators and superconductors}.
\bjtitle{Phys. Rev. B}
\bvolume{90},
\bfpage{165114}
(\byear{2014})
\end{barticle}
\endbibitem

\bibitem{SSG17}
\begin{barticle}
\bauthor{\bsnm{Shiozaki}, \binits{K.}},
\bauthor{\bsnm{Sato}, \binits{M.}},
\bauthor{\bsnm{Gomi}, \binits{K.}}:
\batitle{Topological crystalline materials: General formulation, module structure, and
  wallpaper groups}.
\bjtitle{Phys. Rev. B},
\bvolume{95},
\bfpage{235425}
(\byear{2017})
\end{barticle}
\endbibitem

\bibitem{Sim67}
\begin{barticle}
\bauthor{\bsnm{Simonenko}, \binits{I.B.}}:
\batitle{Convolution type operators in cones}.
\bjtitle{Mat. Sb. (N.S.)}
\bvolume{74}(\bissue{116}),
\bfpage{298}--\blpage{313}
(\byear{1967})
\end{barticle}
\endbibitem

\bibitem{TTM20a}
\begin{barticle}
\bauthor{\bsnm{Takahashi}, \binits{R.}},
\bauthor{\bsnm{Tanaka}, \binits{Y.}},
\bauthor{\bsnm{Murakami}, \binits{S.}}:
\batitle{Bulk-edge and bulk-hinge correspondence in inversion-symmetric
  insulators}.
\bjtitle{Phys. Rev. Res.}
\bvolume{2},
\bfpage{013330}
(\byear{2020})
\end{barticle}
\endbibitem

\bibitem{TTM20b}
\begin{barticle}
\bauthor{\bsnm{Tanaka}, \binits{Y.}},
\bauthor{\bsnm{Takahashi}, \binits{R.}},
\bauthor{\bsnm{Murakami}, \binits{S.}}:
\batitle{Appearance of hinge states in second-order topological insulators via
  the cutting procedure}.
\bjtitle{Phys. Rev. B}
\bvolume{101},
\bfpage{115120}
(\byear{2020})
\end{barticle}
\endbibitem

\bibitem{Thi16}
\begin{barticle}
\bauthor{\bsnm{Thiang}, \binits{G.C.}}:
\batitle{On the {$K$}-theoretic classification of topological phases of
  matter}.
\bjtitle{Ann. Henri Poincar\'e}
\bvolume{17}(\bissue{4}),
\bfpage{757}--\blpage{794}
(\byear{2016})
\end{barticle}
\endbibitem

\bibitem{TZ23}
\begin{barticle}
\bauthor{\bsnm{Thiang}, \binits{G.C.}},
\bauthor{\bsnm{Zhang}, \binits{H.}}:
\batitle{Bulk-interface correspondences for one-dimensional topological
  materials with inversion symmetry}.
\bjtitle{Proc. A.}
\bvolume{479}(\bissue{2270}),
\bfpage{20220675, 22}
(\byear{2023})
\end{barticle}
\endbibitem

\bibitem{LB19}
\begin{barticle}
\bauthor{\bsnm{Trifunovic}, \binits{L.}},
\bauthor{\bsnm{Brouwer}, \binits{P.W.}}:
\batitle{Higher-order bulk-boundary correspondence for topological crystalline
  phases}.
\bjtitle{Phys. Rev. X}
\bvolume{9},
\bfpage{011012}
(\byear{2019})
\end{barticle}
\endbibitem

\bibitem{Witten98}
\begin{barticle}
\bauthor{\bsnm{Witten}, \binits{E.}}:
\bbtitle{D-branes and {K}-theory}.
\bjtitle{JHEP}
\textbf{12}
\bfpage{019}
(\byear{1998})
\end{barticle}
\endbibitem

\end{thebibliography}
\end{document}